\documentclass[a4paper,12pt]{amsart}
\usepackage{amsmath,amssymb,amsfonts,amsthm}
\usepackage{bm}
\usepackage[top=25mm,bottom=25mm,left=25mm,right=25mm]{geometry}
\usepackage{array}
\usepackage{comment}
\usepackage{graphicx}
\usepackage[all]{xy}
\usepackage{relsize}
\usepackage{tikz}
\usetikzlibrary{trees}
\usetikzlibrary{cd,arrows,matrix,backgrounds,shapes,positioning,calc,decorations.markings,decorations.pathmorphing,decorations.pathreplacing}
\usepackage{multirow}
\usepackage{mathtools,leftidx}
\usepackage[colorlinks=true,linkcolor=blue,citecolor=blue,urlcolor=blue]{hyperref}
\makeatletter
\renewcommand{\@defaultbiblabelstyle}[1]{[#1]}
\makeatother

\usetikzlibrary{intersections,arrows,calc,decorations.markings}
\newtheorem{theorem}{Theorem}[section]

\newtheorem{proposition}[theorem]{Proposition}

\newtheorem{corollary}[theorem]{Corollary}
\newtheorem{lemma}[theorem]{Lemma}

\theoremstyle{definition}
\newtheorem{remark}[theorem]{Remark}
\newtheorem{example}[theorem]{Example}
\newtheorem*{example*}{Example}
\newtheorem{definition}[theorem]{Definition}

\makeatletter
 
 \@addtoreset{equation}{section}
\makeatother

\allowdisplaybreaks[4]
\title[Words for generalized Markov numbers]{Words for generalized Markov numbers}
\author{Yasuaki Gyoda}
\address[Yasuaki Gyoda]{Institute for Advanced Research, Nagoya University, Furo-cho, Chikusa-ku, Nagoya-shi, 464-8601, Japan}
\email{ygyoda@math.nagoya-u.ac.jp}

\begin{document}
\begin{abstract}
We construct a word-theoretic framework for generalized Markov numbers, that
is, positive integers appearing in positive integer solutions of the
generalized Markov equation
$x^2+y^2+z^2+k_1yz+k_2zx+k_3xy=(3+k_1+k_2+k_3)xyz$.
For each positive rational slope $t$, we define a word $\omega_t$ by a
recursive rule on a binary tree and realize it geometrically by a line segment
of slope $t$.  Matrix evaluation of $\omega_t$ gives a Markov--monodromy
matrix encoding the generalized Markov number at $t$.
We also show that $\omega_t$ recovers the classical Cohn word by a local
substitution rule, and that the completed word
$\overline{\omega}_t=xyz\omega_t^{-1}$ is related to the generalized Cohn
matrices.
\end{abstract}
\maketitle

\section{Introduction}\label{sec:introduction}

This paper concerns a part of Markov theory arising from indefinite binary
quadratic forms, Diophantine approximation, and the classical Markov equation.
In his two papers \cite{mar1,mar2}, Markov showed that certain extremal
problems in this subject are governed by the integer solutions of
\[
x^2+y^2+z^2=3xyz .
\]
The positive integers which occur in such solutions are called Markov numbers.
They are organized by the Markov tree and play a fundamental role in the
classical Markov and Lagrange spectra.  Standard references for this circle of
ideas include \cite{cusick-flahive,aig}.

One striking feature of the classical theory is that the same numbers can be
described in several different languages.  They arise from Diophantine
approximation and binary quadratic forms, but also from continued fractions,
words, and matrices.  In particular, each rational slope has an associated
Christoffel word.  In the present context, these words are often called Cohn
words because of their relation with the Cohn matrices introduced by Cohn
\cite{cohn}.  A Cohn word is evaluated as a product of two basic matrices, and
the resulting Cohn matrix is a matrixization of a Markov number: the Markov
number can be recovered from the matrix.  This gives a concrete link between
the combinatorics of rational slopes and Markov numbers.  For background on
words, including the relation between Christoffel (Cohn) words and Markov numbers,
see \cite{reutenauer}.

We now pass to the generalized setting.  Fix
$(k_1,k_2,k_3)\in\mathbb Z_{\geq0}^3$ and consider the generalized Markov
equation
\begin{equation}\label{eq:kgm}
 x^2+y^2+z^2+k_1yz+k_2zx+k_3xy
 =(3+k_1+k_2+k_3)xyz.
\end{equation}
The positive integers which occur as components of positive integer solutions
of \eqref{eq:kgm} are called generalized Markov numbers.  The purpose of this
paper is to construct an analogous word-theoretic framework for these numbers.

Generalized Cohn matrices were introduced by Gyoda and Maruyama as
matrixizations of generalized Markov numbers \cite{gyo-maru}.  They play the
role of Cohn matrices in the generalized setting.  However, their recursive
construction is not described by matrix multiplication alone; it also involves
the addition of a constant matrix.  Thus, unlike the classical Cohn matrices,
they are not obtained simply by evaluating words as products of fixed letter
matrices.

A first possible approach would be to generalize Cohn words directly.  Namely,
one might try to attach to each positive rational slope a word whose ordinary
matrix product gives the corresponding generalized Cohn matrix.  The preceding
observation shows why this direct approach does not work.  The desired words
therefore cannot be a straightforward deformation of the classical Cohn words.

The main idea of this paper is to use a different family of words.  Let
$\mathfrak F_3$ be the free group generated by $\{x,y,z\}$.  Consider the
rooted planar binary tree with vertices in $\mathfrak F_3^3$ defined as
follows.  The root is $
 (x,y,z),$ and if a vertex is $(a,b,c)$, then its left and right children are
\[
 (a,bcb^{-1},b),\qquad (b,b^{-1}ab,c),
\]
respectively.  For a vertex $v$ of this tree, we denote by $\omega_v$ the
middle entry of the corresponding triple.  This recursively defined family of
words is the basic combinatorial object used throughout the paper.

These words also admit a geometric realization by line segments of rational
slopes.  In this realization, the word represented by the segment of slope
$t$ is denoted by $\omega_t$.  The word $\omega_t$ serves as a common
source for two constructions: a local substitution rule recovers the classical
Cohn word, while a matrix evaluation gives the Markov--monodromy matrix in the
generalized setting.

We first describe the matrix evaluation. We substitute
parameter-dependent matrices $X,Y,Z$ for the letters $x,y,z$ and write
\[
 M_t:=\omega_t\big|_{x\mapsto X,\;y\mapsto Y,\;z\mapsto Z}.
\]
This matrix $M_t$ is the Markov--monodromy matrix attached to the slope
$t$ for the fixed parameters $k_1,k_2,k_3$.  It is a matrix refinement of
the generalized Markov number at $t$: in the normalization used in this
paper, the generalized Markov number can be recovered from a specified entry of
$M_t$.  Thus the word $\omega_t$ records the combinatorial pattern attached
to the slope $t$, while the dependence on the fixed parameters
$k_1,k_2,k_3$ is carried by the matrices $X,Y,Z$ used in the evaluation.

The other construction relates $\omega_t$ to the classical Cohn word.  The
Cohn word attached to a positive rational slope $t$ is a classical object
independent of the parameters in the generalized Markov equation.  In our
framework, it is obtained from $\omega_t$ by an explicit local substitution
rule.  This gives a direct word-theoretic bridge between the recursively
defined words above and the classical Cohn words.

The relation with generalized Cohn matrices is obtained from an
endpoint-completed version of the same line-segment realization.  We denote the
corresponding word by $\overline{\omega}_t$.  At the level of words, this
completion is especially simple:
\[
 \overline{\omega}_t=xyz\omega_t^{-1}.
\]
Evaluating this word by the same matrices gives
\[
 \overline M_t
 :=\overline{\omega}_t\big|_{x\mapsto X,\;y\mapsto Y,\;z\mapsto Z}.
\]
The precise convention and the proof of the identity
$\overline{\omega}_t=xyz\omega_t^{-1}$ are given in
Section~\ref{sec:gc-extended}.

On the other hand, the same slope and parameters determine a strongly
admissible sequence, and hence a generalized Cohn matrix $C_t$.  The main
comparison theorem states that $\overline M_t$ and $C_t$ are related by a
fixed sign correction.  Thus, although generalized Cohn matrices are not
generated by a purely multiplicative word rule, the endpoint-completed word
$\overline{\omega}_t$ still contains enough information to determine them up
to this fixed correction.

The paper is organized as follows.  Section~\ref{sec:gm-cf} recalls the
combinatorial background for rational slopes, generalized Markov numbers,
characteristic numbers, generalized Markov sequences, and fence posets.
Section~\ref{sec:mm-words} defines the words $\omega_v$ by using the word
tree and proves their geometric realization by line segments of rational
slopes.  Section~\ref{sec:mm-matrices} assigns matrices to these words and
fixes the convention used throughout the paper.  Section~\ref{sec:mm-cohn}
explains how the classical Cohn word is obtained from $\omega_t$.
Section~\ref{sec:gc-extended} introduces generalized Cohn matrices and proves
that they are obtained from $\overline{\omega}_t$.

\subsection*{Acknowledgements}
This work was supported by JSPS KAKENHI Grant Number JP25K17224.
\section{Preliminaries on generalized Markov numbers and fence posets}\label{sec:gm-cf}

\subsection{Farey tree}
We define the \emph{Farey tree} $\mathrm{F}\mathbb T$ as follows:
\begin{itemize}\setlength{\leftskip}{-10pt}
\item[(1)] The root vertex is $\left(\frac{0}{1},\frac{1}{1},\frac{1}{0}\right)$.
\item[(2)] Each vertex $\left(\frac{a}{b},\frac{c}{d},\frac{e}{f}\right)$ has the following two children:
\[\begin{xy}(0,0)*+{\left(\dfrac{a}{b},\dfrac{c}{d},\dfrac{e}{f}\right)}="1",(-30,-15)*+{\left(\dfrac{a}{b},\dfrac{a}{b}\oplus\dfrac{c}{d},\dfrac{c}{d}\right)}="2",(30,-15)*+{\left(\dfrac{c}{d},\dfrac{c}{d}\oplus\dfrac{e}{f},\dfrac{e}{f}\right).}="3", \ar@{-}"1";"2"\ar@{-}"1";"3"
\end{xy}\]
where $\frac{a}{b}\oplus\frac{c}{d}=\frac{a+c}{b+d}$.
\end{itemize}
We write
\[
\mathbb Q_{\ge0}^{\mathrm{ext}}:=\mathbb Q_{\ge0}\cup\{\infty\},
\]
and regard $\frac{0}{1}$ and $\frac{1}{0}$ as the boundary cases.  An element of
$\mathbb Q_{\ge0}^{\mathrm{ext}}$ is represented by a reduced fraction $\frac{p}{q}$ with
$p,q\ge0$, $\gcd(p,q)=1$, where $\frac{1}{0}$ represents $\infty$.

The first few vertices of $\mathrm{F}\mathbb T$ are as follows:
\relsize{+1}
\begin{align*}
\begin{xy}(0,0)*+{\left(\frac{0}{1},\frac{1}{1},\frac{1}{0}\right)}="1",(20,-14)*+{\left(\frac{0}{1},\frac{1}{2},\frac{1}{1}\right)}="2",(20,14)*+{\left(\frac{1}{1},\frac{2}{1},\frac{1}{0}\right)}="3",
(50,-24)*+{\left(\frac{0}{1},\frac{1}{3},\frac{1}{2}\right)}="4",(50,-8)*+{\left(\frac{1}{2},\frac{2}{3},\frac{1}{1}\right)}="5",(50,8)*+{\left(\frac{1}{1},\frac{3}{2},\frac{2}{1}\right)}="6",(50,24)*+{\left(\frac{2}{1},\frac{3}{1},\frac{1}{0}\right)}="7",(85,-28)*+{\left(\frac{0}{1},\frac{1}{4},\frac{1}{3}\right)\cdots}="8",(85,-20)*+{\left(\frac{1}{3},\frac{2}{5},\frac{1}{2}\right)\cdots}="9",(85,-12)*+{\left(\frac{1}{2},\frac{3}{5},\frac{2}{3}\right)\cdots}="10",(85,-4)*+{\left(\frac{2}{3},\frac{3}{4},\frac{1}{1}\right)\cdots}="11",(85,4)*+{\left(\frac{1}{1},\frac{4}{3},\frac{3}{2}\right)\cdots}="12",(85,12)*+{\left(\frac{3}{2},\frac{5}{3},\frac{2}{1}\right)\cdots}="13",(85,20)*+{\left(\frac{2}{1},\frac{5}{2},\frac{3}{1}\right)\cdots}="14",(85,28)*+{\left(\frac{3}{1},\frac{4}{1},\frac{1}{0}\right)\cdots}="15",\ar@{-}"1";"2"\ar@{-}"1";"3"\ar@{-}"2";"4"\ar@{-}"2";"5"\ar@{-}"3";"6"\ar@{-}"3";"7"\ar@{-}"4";"8"\ar@{-}"4";"9"\ar@{-}"5";"10"\ar@{-}"5";"11"\ar@{-}"6";"12"\ar@{-}"6";"13"\ar@{-}"7";"14"\ar@{-}"7";"15"
\end{xy}
\end{align*}
\relsize{-1}
\begin{proposition}[{\cite[Section~3.2]{aig}}]\label{prop:property-farey}
The following hold.
\begin{itemize}\setlength{\leftskip}{-10pt}
    \item[(1)] If $\left(r,t,s\right)$ is a Farey triple, then so are $\left(r,r\oplus t,t\right)$ and $\left(t,t\oplus s,s\right)$. In particular, every vertex in $\mathrm{F}\mathbb T$ is a Farey triple.
    \item[(2)] For every irreducible fraction $t \in (0,\infty)$, there exists a unique Farey triple $F$ in $\mathrm{F}\mathbb T$ such that $t$ is the second entry of $F$.
    \item[(3)] For $\left(r,t,s\right)$ in $\mathrm{F}\mathbb T$, the inequalities $r<t<s$ hold.
\end{itemize}
\end{proposition}
\subsection{\texorpdfstring{$(k_1,k_2,k_3,\sigma)$-GM numbers and characteristic numbers}{(k1,k2,k3;sigma)-GM numbers and characteristic numbers}}
From now on, fix an ordered triple $(k_1,k_2,k_3)\in\mathbb Z_{\geq0}^3$.  The \emph{$(k_1,k_2,k_3)$-generalized Markov equation} is
\[
x^2 + y^2 + z^2 + k_1 yz + k_2 zx + k_3 xy = (3 + k_1 + k_2 + k_3) xyz.
\]
We abbreviate this equation as the $(k_1,k_2,k_3)$-GM equation.

A \emph{$(k_1,k_2,k_3)$-GM triple} is any permutation of a positive integer solution of this equation.  A positive integer appearing in a $(k_1,k_2,k_3)$-GM triple is called a \emph{$(k_1,k_2,k_3)$-GM number}.

To keep track of the position label carried by each $(k_1,k_2,k_3)$-GM number, we use a labeled version of the generalized Markov tree.  Let $\mathfrak S_3$ be the symmetric group on three letters, acting on $\{1,2,3\}$ from the left.  For an ordered triple $(x_1,x_2,x_3)$ and $\tau\in\mathfrak S_3$, we write $\tau(x_1,x_2,x_3)=(x_{\tau(1)},x_{\tau(2)},x_{\tau(3)})$.  Thus, in $\tau(a,b,c)$, the entries $a,b,c$ occupy the positions $\tau^{-1}(1),\tau^{-1}(2),\tau^{-1}(3)$, respectively.  For $\sigma\in\mathfrak S_3$, the \emph{$(k_1,k_2,k_3,\sigma)$-generalized Markov tree}, abbreviated the \emph{$(k_1,k_2,k_3,\sigma)$-GM tree}, is the tree $\mathrm{M}\mathbb T(k_1,k_2,k_3,\sigma)$ defined as follows:
\begin{itemize}\setlength{\leftskip}{-10pt}
\item [(1)] The root vertex is
\[
\big((1,\sigma(1)), (k_{\sigma(2)}+2, \sigma(2)), (1,\sigma(3))\big).
\]
\item [(2)] Every vertex $((a,h),(b,i),(c,j))$ has the following two children:
\[
\begin{xy}
(0,0)*+{((a,h),(b,i),(c,j))}="1",
(-40,-15)*+{\left((a,h),\left(\frac{a^2+k_j ab+b^2}{c},j\right),(b,i)\right)}="2",
(40,-15)*+{\left((b,i),\left(\frac{b^2+k_h bc+c^2}{a},h\right),(c,j)\right).}="3",
\ar@{-}"1";"2"\ar@{-}"1";"3"
\end{xy}
\]
\end{itemize}

\begin{example}\label{ex:CMT(1,2,0,id)}
The first few vertices of $\mathrm{M}\mathbb T(1,2,0;\mathrm{id})$ are as follows:
\begin{align*}
\begin{xy}
(10,0)*+{((1,1),(4,2),(1,3))}="1",
(25,16)*+{((4,2),(21,1),(1,3))}="2",
(25,-16)*+{((1,1),(17,3),(4,2))}="3",
(65,24)*+{((21,1),(121,2),(1,3))}="4",
(65,8)*+{((4,2),(457,3),(21,1))}="5",
(65,-8)*+{((17,3),(373,1),(4,2))}="6",
(65,-24)*+{((1,1),(81,2),(17,3))}="7",
(120,28)*+{((121,2),(703,1),(1,3))\cdots}="8",
(120,20)*+{((21,1),(15082,3),(121,2))\cdots}="9",
(120,12)*+{((457,3),(57121,2),(21,1))\cdots}="10",
(120,4)*+{((4,2),(10033,1),(457,3))\cdots}="11",
(120,-4)*+{((373,1),(8185,3),(4,2))\cdots}="12",
(120,-12)*+{((17,3),(38025,2),(373,1))\cdots}="13",
(120,-20)*+{((81,2),(8227,1),(17,3))\cdots}="14",
(120,-28)*+{((1,1),(386,3),(81,2))\cdots}="15",
\ar@{-}"1";"2"\ar@{-}"1";"3"
\ar@{-}"2";"4"\ar@{-}"2";"5"
\ar@{-}"3";"6"\ar@{-}"3";"7"
\ar@{-}"4";"8"\ar@{-}"4";"9"
\ar@{-}"5";"10"\ar@{-}"5";"11"
\ar@{-}"6";"12"\ar@{-}"6";"13"
\ar@{-}"7";"14"\ar@{-}"7";"15"
\end{xy}
\end{align*}
\end{example}

We recall that the labeled GM tree gives the desired parametrization of
position-labeled generalized Markov triples.
\begin{theorem}[{\cite{gyomatsu}}]\label{thm:all-solution}
Let $(a,b,c)$ be a $(k_1,k_2,k_3)$-GM triple with $b> \max\{a,c\}$, and assume $\tau(a,b,c)$ is a solution to the $(k_1,k_2,k_3)$-GM equation for $\tau\in \mathfrak S_3$. Then there exists a unique pair $\{\sigma,\sigma^{*}\}\subset \mathfrak S_3$ and a unique pair $(v,v^\ast)$ such that $v\in \mathrm{M}\mathbb T(k_1,k_2,k_3,\sigma)$ and $v^\ast\in \mathrm{M}\mathbb T(k_1,k_2,k_3,\sigma^\ast)$, and $v$ and $v^\ast$ are permutations of $((a,\tau^{-1}(1)),(b,\tau^{-1}(2)),(c,\tau^{-1}(3)))$.
\end{theorem}

Before introducing characteristic numbers, we record the following basic
coprimality property.
\begin{proposition}[{\cite[Corollary~8]{gyomatsu}}]\label{relatively-prime-generalized}
For any $(k_1,k_2,k_3)$-GM triple $(a,b,c)$, any two of $a,b,c$ are coprime.
\end{proposition}
Let $(r,t,s)$ be a Farey triple, and write the corresponding vertex of
$\mathrm{M}\mathbb T(k_1,k_2,k_3,\sigma)$ as
\[
((m_r,i_r),(m_t,i_t),(m_s,i_s)).
\]
Here $m_r,m_t,m_s$ are the GM numbers attached to $r,t,s$, respectively,
and $i_r,i_t,i_s\in\{1,2,3\}$ are their position labels.  The
\emph{characteristic number} associated with $t$ is the unique integer $u_t$
satisfying
\[
\begin{cases}
m_r u_t \equiv m_s \pmod{m_t},\\
0<u_t<m_t.
\end{cases}
\]
The uniqueness follows from the pairwise coprimality in
Proposition~\ref{relatively-prime-generalized}.  Since the Farey triple
$(r,t,s)$ is uniquely determined by its middle component $t$ by
Proposition~\ref{prop:property-farey}(2), the integer $u_t$ depends only on
$t$.

We also use the coefficient selected by the position label of $m_t$.  With
$i_t$ as above, set
\[
k_t:=k_{i_t}.
\]

\subsection{Generalized Markov sequences from line segments}\label{subsec:gm-sequences-from-line-segments}
Let $\widetilde{\mathbb{R}^{2}}$ be the modified lattice.  We regard
$\widetilde{\mathbb{R}^{2}}$ as the planar graph obtained by adding, to the
square grid, all line segments of slope $-1$ connecting adjacent lattice
vertices.  For a positive reduced fraction
$t=\frac{p}{q}\in(0,\infty)\cap\mathbb Q$, let $L_t$ be the oriented segment
from $(0,0)$ to $(q,p)$.  Figure~\ref{fig:ex-presnakegraph} shows the case
$t=\frac{2}{5}$.

\begin{figure}[ht]
    \centering
    \begin{tikzpicture}[x=1.2cm,y=1.2cm]
    \definecolor{segmentred}{RGB}{255,0,0}
    \tikzset{
      figgrid/.style={black,line width=0.7pt,line cap=butt,line join=miter},
      figsegment/.style={draw=segmentred,line width=1.0pt,line cap=round,line join=round}
    }
    \draw[figgrid] (0,0) -- (3,0);
    \draw[figgrid] (0,1) -- (5,1);
    \draw[figgrid] (2,2) -- (5,2);
    \draw[figgrid] (0,0) -- (0,1);
    \draw[figgrid] (1,0) -- (1,1);
    \draw[figgrid] (2,0) -- (2,2);
    \draw[figgrid] (3,0) -- (3,2);
    \draw[figgrid] (4,1) -- (4,2);
    \draw[figgrid] (5,1) -- (5,2);
    \draw[figgrid] (0,1) -- (1,0);
    \draw[figgrid] (1,1) -- (2,0);
    \draw[figgrid] (2,1) -- (3,0);
    \draw[figgrid] (2,2) -- (3,1);
    \draw[figgrid] (3,2) -- (4,1);
    \draw[figgrid] (4,2) -- (5,1);
    \draw[figsegment] (0,0) -- (5,2);
    \end{tikzpicture}
    \caption{Line segment $L_{\frac{2}{5}}$}
    \label{fig:ex-presnakegraph}
\end{figure}
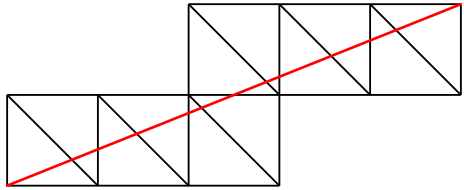

We define the \emph{$(k_1,k_2,k_3,\sigma)$-generalized Markov sequence}
\[
s^{\circ}_{k_1,k_2,k_3,\sigma}(t)=(a_1,\dots,a_\ell)
\]
by a sign assignment along $L_t$.  The segment is oriented from left to right.  Write $L_t^\circ$ for $L_t$ with its endpoints removed.  A triangle is counted as crossed when the open segment $L_t^\circ$ meets the interior of the triangle.

First, assign a sign to each right-angled triangle crossed by $L_t$.  The sign is $-$ for the leftmost crossed triangle and for any crossed triangle whose left-hand side, cut by $L_t$, is a quadrilateral.  All other crossed triangles receive the sign $+$; see Figures~\ref{fig:minus-righttriangles} and~\ref{fig:plus-righttriangles}.

\begin{figure}[ht]
    \centering
\begin{tikzpicture}
\draw (0,0) -- (1,0) -- (0,1) -- cycle;
\draw[red] (0,0) -- (0.6,0.6);
\end{tikzpicture}
\hspace{0.7cm}
\begin{tikzpicture}[baseline=0mm]
\draw (0,0) -- (1,0) -- (0,1) -- cycle;
\draw[red] (0,-0.1) -- (1.1,0.4);
\end{tikzpicture}\hspace{0.4cm}
\begin{tikzpicture}[baseline=0mm]
\draw (1,0) -- (1,1) -- (0,1) -- cycle;
\draw[red] (1.1,0.6) -- (0.6,0.3);
\end{tikzpicture}
    \caption{Right-angled triangles with $-$}
    \label{fig:minus-righttriangles}
\end{figure}

\begin{figure}[ht]
    \centering
\rotatebox{180}{\begin{tikzpicture}[baseline=0mm]
\draw (1,0) -- (1,1) -- (0,1) -- cycle;
\draw[red] (1.1,0.6) -- (0.6,0.3);
\end{tikzpicture}}
\hspace{0.3cm}
\rotatebox{180}{
\begin{tikzpicture}[baseline=0mm]
\draw (0,0) -- (1,0) -- (0,1) -- cycle;
\draw[red] (0,-0.1) -- (1.1,0.4);
\end{tikzpicture}}\hspace{0.5cm}
\rotatebox{180}{\begin{tikzpicture}
\draw (0,0) -- (1,0) -- (0,1) -- cycle;
\draw[red] (0,0) -- (0.6,0.6);
\end{tikzpicture}}
    \caption{Right-angled triangles with $+$}
    \label{fig:plus-righttriangles}
\end{figure}

Second, assign a sign to each edge whose interior intersects $L_t$.  A horizontal edge contributes $k_{\sigma(1)}$ copies of that sign, a diagonal edge contributes $k_{\sigma(2)}$ copies, and a vertical edge contributes $k_{\sigma(3)}$ copies.  The sign is $-$ if the midpoint of the edge is not on the right-hand side of the oriented segment $L_t$, and it is $+$ if the midpoint is on the right-hand side; the phrase ``not on the right-hand side'' includes the case where the midpoint lies on $L_t$ itself.  See Figures~\ref{fig:minus-sign-edge} and~\ref{fig:plus-sign-edge}.

\begin{figure}[ht]
    \centering
\begin{tikzpicture}[baseline=0mm]
\draw (-0.5,0) -- (0.5,0);
\fill (0,0) circle (1.5pt);
\draw[red] (-0.5,-0.3) -- (0.5,0.3);
\end{tikzpicture}
\hspace{0.7cm}
\begin{tikzpicture}[baseline=0mm]
\draw (-0.5,0) -- (0.5,0);
\fill (0,0) circle (1.5pt);
\draw[red] (-0.5,-0.3) -- (0.5,0.1);
\end{tikzpicture}
\hspace{0.7cm}
\begin{tikzpicture}[baseline=0mm]
\draw (0,-0.5) -- (0,0.5);
\fill (0,0) circle (1.5pt);
\draw[red] (-0.5,-0.3) -- (0.5,0.3);
\end{tikzpicture}\hspace{0.6cm}
\begin{tikzpicture}[baseline=0mm]
\draw (0,-0.5) -- (0,0.5);
\fill (0,0) circle (1.5pt);
\draw[red] (-0.5,-0.3) -- (0.5,0);
\end{tikzpicture}\hspace{0.6cm}
\begin{tikzpicture}[baseline=0mm]
\draw (-0.5,0.5) -- (0.5,-0.5);
\fill (0,0) circle (1.5pt);
\draw[red] (-0.5,-0.3) -- (0.5,0.3);
\end{tikzpicture}\hspace{0.6cm}
\begin{tikzpicture}[baseline=0mm]
\draw (-0.5,0.5) -- (0.5,-0.5);
\fill (0,0) circle (1.5pt);
\draw[red] (-0.5,-0.3) -- (0.5,0);
\end{tikzpicture}
    \caption{Edges with $-$}
    \label{fig:minus-sign-edge}
\end{figure}

\begin{figure}[ht]
    \centering
\begin{tikzpicture}[baseline=0mm]
\draw (-0.5,0) -- (0.5,0);
\fill (0,0) circle (1.5pt);
\draw[red] (-0.5,-0.1) -- (0.5,0.3);
\end{tikzpicture}
\hspace{0.7cm}
\begin{tikzpicture}[baseline=0mm]
\draw (0,-0.5) -- (0,0.5);
\fill (0,0) circle (1.5pt);
\draw[red] (-0.5,-0) -- (0.5,0.3);
\end{tikzpicture}\hspace{0.6cm}
\begin{tikzpicture}[baseline=0mm]
\draw (-0.5,0.5) -- (0.5,-0.5);
\fill (0,0) circle (1.5pt);
\draw[red] (-0.5,-0) -- (0.5,0.3);
\end{tikzpicture}
    \caption{Edges with $+$}
    \label{fig:plus-sign-edge}
\end{figure}

Finally, record the sign of a triangle when $L_t$ passes through its interior, and record the sign of an edge when $L_t$ crosses that edge.  These events are ordered by their positions along the oriented segment.  Compress the resulting sign string into maximal consecutive blocks of equal signs.  If the block lengths are $a_1,\dots,a_\ell$, set
\[
s^{\circ}_{k_1,k_2,k_3,\sigma}(t):=(a_1,\dots,a_\ell).
\]

The following example will also be used after the fence-poset notation is introduced.

\begin{example}\label{ex:gm-sequence-two-fifths-before-fence}
Take $(k_1,k_2,k_3)=(1,2,0)$, $\sigma=\mathrm{id}$, and $t=\frac{2}{5}$.  The segment $L_t$ carries the signs shown in Figure~\ref{fig:Lt-signed}.  Since horizontal edges contribute one copy of their sign, diagonal edges contribute two copies, and vertical edges contribute no copies, the sign string obtained by reading from left to right is
\[
\underbrace{-\,-\,-\,-}_{4}\,
\underbrace{+\,+\,+}_{3}\,
\underbrace{-}_{1}\,
\underbrace{+\,+\,+\,+}_{4}\,
\underbrace{-\,-\,-\,-\,-}_{5}\,
\underbrace{+}_{1}\,
\underbrace{-\,-\,-}_{3}\,
\underbrace{+\,+\,+\,+}_{4}.
\]
Therefore the corresponding $(1,2,0,\mathrm{id})$-generalized Markov sequence is
\[
 s^{\circ}_{1,2,0,\mathrm{id}}\left(\frac{2}{5}\right)
 =(4,3,1,4,5,1,3,4).
\]
This illustrates the sign assignment before the endpoint completion used later.

\begin{figure}[ht]
\centering
\begin{tikzpicture}[x=0.01cm,y=0.01cm,scale=0.5]
\definecolor{annred}{RGB}{255,0,0}
\tikzset{
  gridline/.style={black,line width=0.7pt,line cap=butt,line join=miter},
  gammaL/.style={draw=annred,line width=1.1pt,line cap=round,line join=round},
  midpoint/.style={circle,fill=gray,inner sep=0pt,minimum size=1pt},
  enddot/.style={circle,fill=black,inner sep=0pt,minimum size=1.8pt},
  triwhite/.style={circle,fill=white,draw=none,inner sep=0.45pt,minimum size=8.0pt,outer sep=0pt},
  edgewhite/.style={circle,fill=white,draw=none,inner sep=0.35pt,minimum size=7.6pt,outer sep=0pt},
  trisign/.style={text=annred,font=\scriptsize\bfseries,inner sep=0.45pt,minimum size=5.2pt,outer sep=0pt},
  edgesign/.style={text=annred,font=\scriptsize\bfseries,inner sep=0.35pt,minimum size=4.9pt,outer sep=0pt}
}
\begin{scope}[shift={(300,170)}, x={(260,0)}, y={(0,260)}]
\begin{scope}
  \clip (0,0) rectangle (5,2);
  \foreach \i in {0,...,4}{
    \foreach \j in {0,1}{
      \draw[gridline] (\i,{\j+1}) -- ({\i+1},\j);
    }
  }
  \foreach \i in {0,...,5}{
    \draw[gridline] (\i,0) -- (\i,2);
  }
  \foreach \j in {0,1,2}{
    \draw[gridline] (0,\j) -- (5,\j);
  }
  \foreach \i in {0,...,4}{
    \foreach \j in {0,1,2}{
      \node[midpoint] at ({\i+0.5},\j) {};
    }
  }
  \foreach \i in {0,...,5}{
    \foreach \j in {0,1}{
      \node[midpoint] at (\i,{\j+0.5}) {};
    }
  }
  \foreach \i in {0,...,4}{
    \foreach \j in {0,1}{
      \node[midpoint] at ({\i+0.5},{\j+0.5}) {};
    }
  }
\end{scope}
\coordinate (A) at (0,0);
\coordinate (B) at (5,2);
\draw[gammaL] (A) -- (B);
\foreach \x/\y/\sig in {0.20/0.18/-,0.78/0.78/-,1.20/0.22/+,1.73/0.58/-,2.18/0.55/+,2.70/0.70/+,2.30/1.18/-,2.78/1.40/-,3.05/1.18/+,3.82/1.80/-,4.18/1.60/+,4.75/1.85/+}{%
  \node[triwhite] at (\x,\y) {};
}
\foreach \x/\y/\sig in {0.43/0.40/-,0.65/0.40/-,1.35/0.48/+,1.57/0.48/+,2.36/0.54/+,2.58/0.54/+,2.58/1.00/-,2.32/1.47/-,2.54/1.47/-,3.46/1.40/-,3.68/1.40/-,4.42/1.60/+,4.64/1.60/+}{%
  \node[edgewhite] at (\x,\y) {};
}
\foreach \x/\y/\sig in {0.20/0.18/-,0.78/0.78/-,1.20/0.22/+,1.73/0.58/-,2.18/0.55/+,2.70/0.70/+,2.30/1.18/-,2.78/1.40/-,3.05/1.18/+,3.82/1.80/-,4.18/1.60/+,4.75/1.85/+}{%
  \node[trisign] at (\x,\y) {$\sig$};
}
\foreach \x/\y/\sig in {0.43/0.40/-,0.65/0.40/-,1.35/0.48/+,1.57/0.48/+,2.36/0.54/+,2.58/0.54/+,2.58/1.00/-,2.32/1.47/-,2.54/1.47/-,3.46/1.40/-,3.68/1.40/-,4.42/1.60/+,4.64/1.60/+}{%
  \node[edgesign] at (\x,\y) {$\sig$};
}
\node[enddot] at (A) {};
\node[enddot] at (B) {};
\end{scope}
\end{tikzpicture}
\caption{The segment $L_t$ for $t=\frac{2}{5}$ with the signs assigned by the triangle-crossing and edge-crossing rules.}
\label{fig:Lt-signed}
\end{figure}
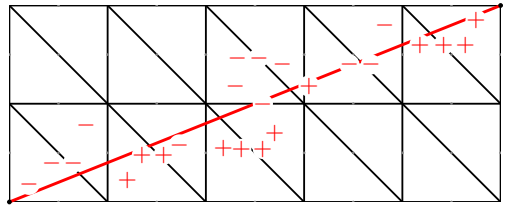
\end{example}

\subsection{Fence posets and order ideals}\label{subsec:fence-posets}
We record the poset notation associated with finite integer sequences.  This notation will be used later to express the entries of Markov-monodromy matrices in terms of order ideals.

Let $(P,\preceq)$ be a finite poset.  For $x,y\in P$, we write $x\lessdot y$ if $x\prec y$ and there is no $z\in P$ such that $x\prec z\prec y$.  This relation is called the \emph{cover relation}.  The \emph{Hasse diagram} of $P$ is the graph whose vertex set is $P$ and whose edges are the unordered pairs $\{x,y\}$ such that either $x\lessdot y$ or $y\lessdot x$.  It is drawn so that the larger element is placed above the smaller one.

\begin{definition}
A finite poset is called a \emph{fence poset} if the underlying undirected graph of its Hasse diagram is either empty or a path; in particular, the empty poset and the one-element poset are allowed.
\end{definition}

Let $S=(a_1,\dots,a_n)$ be a finite sequence of positive integers.  Put
\[
s_j:=\sum_{i=1}^j a_i \qquad (j=1,\dots,n),
\]
and set $s_0:=0$.  We define a fence poset
\[
P_S=P_{(a_1,\dots,a_n)}:=(\{1,2,\dots,s_n-1\},\preceq)
\]
as follows.  For $x\in \{1,\dots,s_n-1\}$, choose the unique $j\in\{1,\dots,n\}$ such that
\[
s_{j-1}\leq x<s_j,
\]
and set
\[
\varepsilon(x):=(-1)^{j-1}.
\]
For adjacent labels $x,x+1$ with $1\leq x<s_n-1$, the cover relations are defined by
\[
\begin{cases}
x\lessdot x+1 & \text{if }\varepsilon(x)=1,\\
x+1\lessdot x & \text{if }\varepsilon(x)=-1.
\end{cases}
\]
Thus the direction of the Hasse diagram alternates whenever one passes from one block of the sequence $S$ to the next.  Notice that the last label is $s_n-1$, not $s_n$; equivalently, the numbers of edges in the successive blocks are
\[
a_1-1,\ a_2,\ a_3,\dots,\ a_{n-1},\ a_n-1.
\]

\begin{definition}
Let $(P,\preceq)$ be a fence poset.  A subset $I\subset P$ is called an \emph{order ideal} of $P$ if, whenever $x\in I$, $y\in P$, and $y\preceq x$, one has $y\in I$.  We denote by $\mathcal J(P)$ the set of order ideals of $P$.  The empty set is also regarded as an order ideal.
\end{definition}

For a fence poset $P$, we write
\[
N(P):=\#\mathcal J(P).
\]
For a finite sequence of positive integers $(a_1,\dots,a_n)$, we define
\[
N(a_1,\dots,a_n):=N(P_{(a_1,\dots,a_n)}).
\]
We also use the convention $N()=1$ for the empty sequence.  More generally,
\[
N(a_i,\dots,a_j):=1\qquad (i>j).
\]
\begin{example}
Let $S=(2,2)$.  Then $s_2=4$, so $P_S$ has the three elements $1,2,3$.  The cover relations are
\[
1\lessdot 2,\qquad 3\lessdot 2.
\]
Equivalently, the Hasse diagram of $P_S$ is
\begin{center}
\begin{tikzpicture}[scale=1]
\node (two) at (0,1.2) {$2$};
\node (one) at (-0.8,0) {$1$};
\node (three) at (0.8,0) {$3$};
\draw (one) -- (two) -- (three);
\end{tikzpicture}
\end{center}
where larger elements are placed higher.  Its order ideals are
\[
\emptyset,\quad \{1\},\quad \{3\},\quad \{1,3\},\quad \{1,2,3\}.
\]
Hence $N(2,2)=5$.
\end{example}

We now relate the sequence obtained from the line segment $L_t$ to the
generalized Markov number attached to the same slope.

\begin{theorem}[{\cite{bana-gyo}}]\label{thm:gm-sequence-counts-gm}
For any $(k_1,k_2,k_3)\in \mathbb Z_{\geq 0}^3$, $\sigma\in \mathfrak S_3$, and $t\in (0,\infty)\cap \mathbb Q$, one has
\[
N(s^{\circ}_{k_1,k_2,k_3,\sigma}(t))=m_t,
\]
where $m_t$ is a $(k_1,k_2,k_3)$-GM number parameterized by $t$ in $\mathrm M\mathbb T(k_1,k_2,k_3,\sigma)$.
\end{theorem}

\begin{example}\label{ex:theorem-gm-sequence-two-fifths}
We continue Example~\ref{ex:gm-sequence-two-fifths-before-fence}.  Theorem~\ref{thm:gm-sequence-counts-gm} gives
\[
 N\left(4,3,1,4,5,1,3,4\right)=m_{\frac{2}{5}}.
\]
We compute this order-ideal number by using continued fractions.  By \cite[Corollary~13]{banaian-kmarkov}, it is the numerator of the continued fraction attached to the same shape.  Hence
\[
 N\left(4,3,1,4,5,1,3,4\right)
 =\operatorname{num}[4;3,1,4,5,1,3,4].
\]
Moreover,
\[
[4;3,1,4,5,1,3,4]
=4+\cfrac{1}{3+\cfrac{1}{1+\cfrac{1}{4+\cfrac{1}{5+\cfrac{1}{1+\cfrac{1}{3+\cfrac{1}{4}}}}}}}
=\frac{8227}{1930}.
\]
Therefore
\[
 N\left(4,3,1,4,5,1,3,4\right)=8227.
\]
The corresponding Farey triple is $(\frac{1}{3},\frac{2}{5},\frac{1}{2})$.  In the tree $\mathrm M\mathbb T(1,2,0,\mathrm{id})$, the corresponding vertex is
\[
 ((81,2),(8227,1),(17,3)).
\]
Thus $m_{\frac{2}{5}}=8227$.  In particular, the positioned triple
\[
 (x_1,x_2,x_3)=(8227,81,17)
\]
is a positive integer solution of the $(1,2,0)$-GM equation.  Indeed,
\[
\begin{aligned}
&8227^2+81^2+17^2+81\cdot17+2\cdot17\cdot8227\\
&\quad=67683529+6561+289+1377+279718\\
&\quad=67971474\\
&\quad=6\cdot8227\cdot81\cdot17.
\end{aligned}
\]
This gives a concrete instance in which the value computed from the sequence above is a GM number.
\end{example}

\section{Words from rational slopes and Markov-monodromy matrices}\label{sec:mm-words}
\subsection{The word tree and its geometric realization}

Let $\mathfrak F_3$ be the free group generated by $\{x,y,z\}$.  We consider the following rooted planar binary tree with vertices in $\mathfrak F_3^3$:
\begin{itemize}\setlength{\leftskip}{-10pt}
\item[(1)] The root vertex is $(x,y,z)\in \mathfrak F_3^3$.
\item[(2)] A vertex $(a,b,c)$ has the following two children:
\[\begin{xy}(0,0)*+{(a,b,c)}="1",(-20,-10)*+{(a,bcb^{-1},b)}="2",(20,-10)*+{(b,b^{-1}ab,c).}="3", \ar@{-}"1";"2"\ar@{-}"1";"3"
\end{xy}\]
\end{itemize}
We call this tree the \emph{word tree} and denote it by $\mathrm{W}\mathbb T$.
Its vertices provide the recursive source of the words considered below.
The first few vertices of $\mathrm{W}\mathbb T$ are as follows:
\[
\resizebox{0.98\textwidth}{!}{$\begin{xy}(-8,0)*+{(x,y,z)}="0",(0,20)*+{\left(y,y^{-1}xy,z\right)}="1",(0,-20)*+{\left(x,yzy^{-1},y\right)}="1'",(25,30)*+{\left(y^{-1}xy,y^{-1}x^{-1}yxy,z\right)}="20",(25,10)*+{\left(y,y^{-1}xyzy^{-1}x^{-1}y,y^{-1}xy\right)}="21",(25,-10)*+{\left(yzy^{-1},yz^{-1}y^{-1}xyzy^{-1},y\right)}="22",(25,-30)*+{\left(x,yzyz^{-1}y^{-1},yzy^{-1}\right)}="23",(103,35)*+{\left(y^{-1}x^{-1}yxy,y^{-1}x^{-1}y^{-1}xyxy,z\right)}="30",(103,25)*+{\left(y^{-1}xy,y^{-1}x^{-1}yxyzy^{-1}x^{-1}y^{-1}xy,y^{-1}x^{-1}yxy\right)}="31",(103,15)*+{\left(y^{-1}xyzy^{-1}x^{-1}y,y^{-1}xyz^{-1}y^{-1}x^{-1}yxyzy^{-1}x^{-1}y,y^{-1}xy\right)}="32",(103,5)*+{\left(y,y^{-1}xyzy^{-1}xyz^{-1}y^{-1}x^{-1}y,y^{-1}xyzy^{-1}x^{-1}y\right)}="33",(103,-5)*+{\left(yz^{-1}y^{-1}xyzy^{-1},yz^{-1}y^{-1}x^{-1}yzy^{-1}xyzy^{-1},y\right)}="34",(103,-15)*+{\left(yzy^{-1},yz^{-1}y^{-1}xyzyz^{-1}y^{-1}x^{-1}yzy^{-1},yz^{-1}y^{-1}xyzy^{-1}\right)}="35",(103,-25)*+{\left(yzyz^{-1}y^{-1},yzy^{-1}z^{-1}y^{-1}xyzyz^{-1}y^{-1},yzy^{-1}\right)}="36",(103,-35)*+{\left(x,yzyzy^{-1}z^{-1}y^{-1},yzyz^{-1}y^{-1}\right)}="37",\ar@{-}"0";"1"\ar@{-}"0";"1'"\ar@{-}"1";"20"\ar@{-}"1";"21"\ar@{-}"1'";"22"\ar@{-}"1'";"23"\ar@{-}"20";"30"\ar@{-}"20";"31"\ar@{-}"21";"32"\ar@{-}"21";"33"\ar@{-}"22";"34"\ar@{-}"22";"35"\ar@{-}"23";"36"\ar@{-}"23";"37"
\end{xy}$}.
\]
\begin{remark}
Even if $a,b,c$ are reduced words, the products $bcb^{-1}$ and $b^{-1}ab$ need not be reduced.
\end{remark}
We next describe the geometric realization of the words attached to slopes.  For $t\in(0,\infty)\cap\mathbb Q$, let $\Omega(t)$ be the element of $\mathfrak F_3$ located at the unique vertex $(r,t,s)$ of the Farey tree whose middle component is $t$.  For the boundary cases, set $\Omega(\frac{0}{1})=x$ and $\Omega(\frac{1}{0})=z$.

We construct a reduced representative $\omega_t$ of $\Omega(t)$.  Set $\omega_{\frac{0}{1}}=x$ and $\omega_{\frac{1}{0}}=z$.  For a positive reduced fraction $t=\frac{p}{q}\in(0,\infty)\cap\mathbb Q$, consider the segment $L_t$ from $(0,0)$ to $(q,p)$ in the modified lattice $\widetilde{\mathbb{R}^{2}}$, oriented from left to right. See Figure~\ref{fig:ex-presnakegraph}.

The word $\omega_t$ is obtained by reading the edges crossed by $L_t$.  A horizontal (resp., diagonal, vertical) edge contributes the letter $x^{\pm1}$ (resp., $y^{\pm1}$, $z^{\pm1}$).  The exponent is determined by the side on which the edge midpoint lies: if the midpoint is not on the right-hand side of the oriented segment $L_t$, including the case where it lies on $L_t$ itself, we use $x$, $y$, or $z$; if it is on the right-hand side, we use $x^{-1}$, $y^{-1}$, or $z^{-1}$.  The letters are then read in the order in which the corresponding edge-crossing events occur along the oriented segment.  See Figures~\ref{fig:mm-minus-edge} and~\ref{fig:mm-plus-edge}.

\begin{figure}[ht]
    \centering
\begin{tikzpicture}[baseline=0mm]
\draw (-0.5,0) -- (0.5,0);
\fill (0,0) circle (1.5pt);
\draw[red] (-0.5,-0.3) -- (0.5,0.3);
\end{tikzpicture}
\hspace{0.7cm}
\begin{tikzpicture}[baseline=0mm]
\draw (-0.5,0) -- (0.5,0);
\fill (0,0) circle (1.5pt);
\draw[red] (-0.5,-0.3) -- (0.5,0.1);
\end{tikzpicture}
\hspace{0.7cm}
\begin{tikzpicture}[baseline=0mm]
\draw (-0.5,0.5) -- (0.5,-0.5);
\fill (0,0) circle (1.5pt);
\draw[red] (-0.5,-0.3) -- (0.5,0.3);
\end{tikzpicture}\hspace{0.6cm}
\begin{tikzpicture}[baseline=0mm]
\draw (-0.5,0.5) -- (0.5,-0.5);
\fill (0,0) circle (1.5pt);
\draw[red] (-0.5,-0.3) -- (0.5,0);
\end{tikzpicture}
\hspace{0.7cm}
\begin{tikzpicture}[baseline=0mm]
\draw (0,-0.5) -- (0,0.5);
\fill (0,0) circle (1.5pt);
\draw[red] (-0.5,-0.3) -- (0.5,0.3);
\end{tikzpicture}\hspace{0.6cm}
\begin{tikzpicture}[baseline=0mm]
\draw (0,-0.5) -- (0,0.5);
\fill (0,0) circle (1.5pt);
\draw[red] (-0.5,-0.3) -- (0.5,0);
\end{tikzpicture}
    \caption{Edges with $x,y$ and $z$}
    \label{fig:mm-minus-edge}
\end{figure}

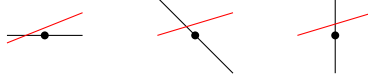
\begin{figure}[ht]
    \centering
\begin{tikzpicture}[baseline=0mm]
\draw (-0.5,0) -- (0.5,0);
\fill (0,0) circle (1.5pt);
\draw[red] (-0.5,-0.1) -- (0.5,0.3);
\end{tikzpicture}
\hspace{0.7cm}
\begin{tikzpicture}[baseline=0mm]
\draw (-0.5,0.5) -- (0.5,-0.5);
\fill (0,0) circle (1.5pt);
\draw[red] (-0.5,-0) -- (0.5,0.3);
\end{tikzpicture}\hspace{0.7cm}
\begin{tikzpicture}[baseline=0mm]
\draw (0,-0.5) -- (0,0.5);
\fill (0,0) circle (1.5pt);
\draw[red] (-0.5,-0) -- (0.5,0.3);
\end{tikzpicture}
    \caption{Edges with $x^{-1},y^{-1}$ and $z^{-1}$}
    \label{fig:mm-plus-edge}
\end{figure}

\begin{example}\label{ex:continued-fraction-Markov}
Let $t=\frac{2}{5}$. The letters assigned to the edges in $L_{\frac{2}{5}}$ are given as in Figure~\ref{fig:ex-signed-presnakegraph}, and we have $\omega_{\frac{2}{5}}=yzy^{-1}z^{-1}y^{-1}xyzyz^{-1}y^{-1}$.
\begin{figure}[ht]
    \centering
    \includegraphics[scale=0.08]{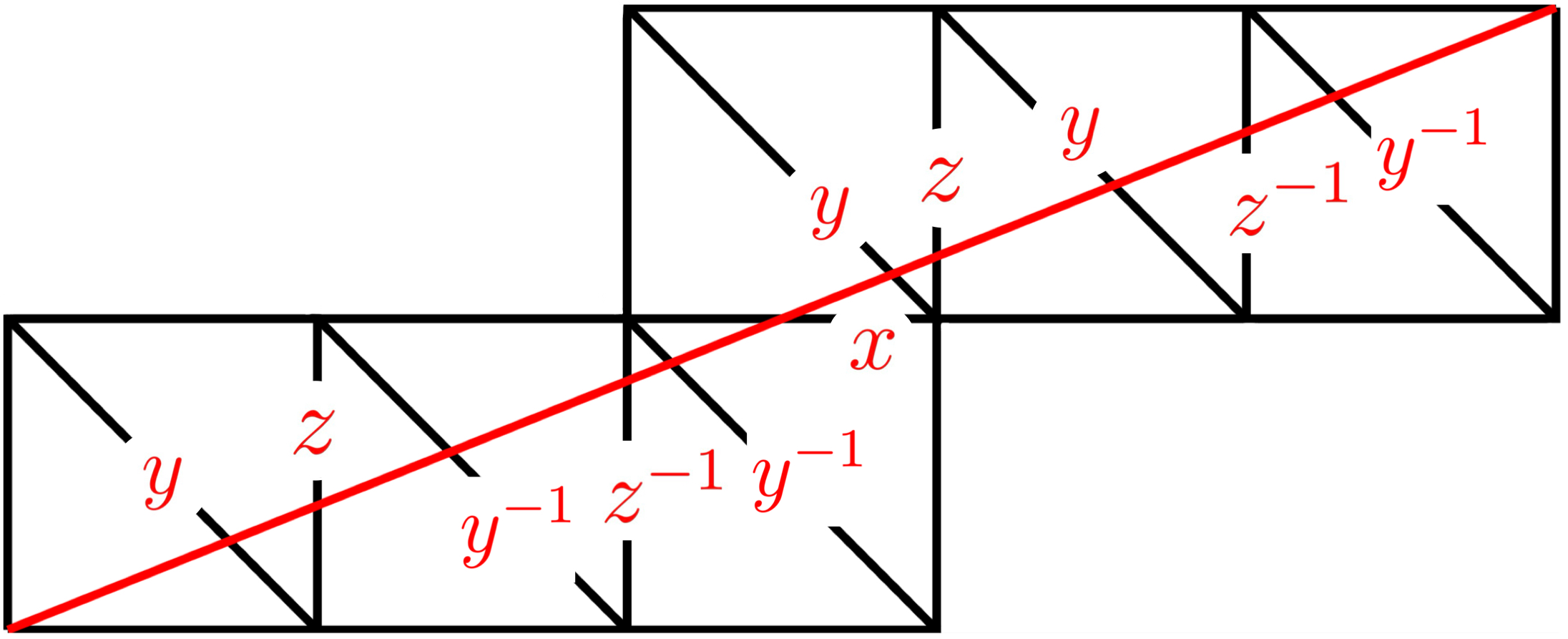}
    \caption{Letters assigned to edges in $L_{\frac{2}{5}}$}
    \label{fig:ex-signed-presnakegraph}
\end{figure}
Thus $\omega_{\frac{2}{5}}$ is the word attached to the slope $\frac{2}{5}$.
\end{example}
No two consecutive crossed edges have the same slope, and hence no adjacent inverse letters are created.  Thus we have the following proposition.
\begin{proposition}\label{prop:reduced-word}
For any irreducible fraction $t\in\mathbb Q_{\ge0}^{\mathrm{ext}}$, the word $\omega_{t}$ is a reduced word.
\end{proposition}

\begin{proof}
The boundary cases $t=\frac{0}{1}$ and $t=\frac{1}{0}$ are immediate because $\omega_{\frac{0}{1}}=x$ and $\omega_{\frac{1}{0}}=z$.  Assume $t\in(0,\infty)\cap\mathbb Q$.  Consecutive edges crossed by $L_t$ have distinct slopes.  Hence two consecutive letters in $\omega_t$ are never powers of the same generator, so no adjacent inverse pair can occur.
\end{proof}

The symmetry of the construction also gives the following useful form of $\omega_t$.
\begin{proposition}\label{prop:symmetric-decomposition}
For any irreducible fraction $t\in\mathbb Q_{\ge0}^{\mathrm{ext}}$, there exist a word $w$ and a letter $\alpha\in \{x,y,z\}$ such that $\omega_t=w\alpha w^{-1}$.
\end{proposition}

\begin{proof}
The boundary cases are immediate.  Assume $t\in(0,\infty)\cap\mathbb Q$.  The segment $L_t$ is centrally symmetric.  The crossed edges occur in symmetric pairs around the middle crossing, and the two letters in each pair are mutually inverse under the chosen orientation convention.  Reading from the two ends toward the center therefore gives $\omega_t=w\alpha w^{-1}$ for the letter $\alpha$ attached to the middle crossing.
\end{proof}

We call the decomposition $\omega_t=w\alpha w^{-1}$ in Proposition~\ref{prop:symmetric-decomposition} the \emph{symmetric decomposition of $\omega_t$}.

We use the following convention for isomorphisms of rooted planar binary trees.
\begin{definition}
Two rooted planar binary trees are said to be \emph{isomorphic} if there exists a graph isomorphism between them which sends the root to the root and sends the left, respectively right, child of each vertex to the left, respectively right, child of its image.  Such an isomorphism, if it exists, is necessarily unique.
\end{definition}
\begin{theorem}\label{thm:ft-wt}
Regard each $\omega_t$ as an element of $\mathfrak F_3$, and define
\[
\Phi:V(\mathrm{F}\mathbb T)\longrightarrow V(\mathrm{W}\mathbb T),
\qquad
\Phi(r,t,s)=(\omega_r,\omega_t,\omega_s).
\]
Then $\Phi$ is the isomorphism of rooted planar binary trees from $\mathrm{F}\mathbb T$ to $\mathrm{W}\mathbb T$.
\end{theorem}
We record the following immediate consequence.
\begin{corollary}
For $t\in(0,\infty)\cap\mathbb Q$, the word $\omega_t$ is the reduced representative of $\Omega(t)$ obtained from the line segment $L_t$; for the boundary cases, the representatives are the conventions $\omega_{\frac{0}{1}}=x$ and $\omega_{\frac{1}{0}}=z$.
\end{corollary}

\begin{proof}
By Theorem~\ref{thm:ft-wt}, the word $\omega_t$ represents $\Omega(t)$, and by Proposition~\ref{prop:reduced-word} it is reduced.
\end{proof}

The proof of Theorem~\ref{thm:ft-wt} uses the following two auxiliary facts: one algebraic fact and one geometric fact.

\begin{lemma}\label{prop:presnake-relation-W}
For a Farey triple $(r,t,s)$, we have $\Omega(t)=\Omega(r)^{-1}xyz\Omega(s)^{-1}$.
\end{lemma}
\begin{proof}
If $(r,t,s)=\left(\frac{0}{1},\frac{1}{1},\frac{1}{0}\right)$, we have $\Omega(r)=x,\Omega(t)=y,\Omega(s)=z$ by definition, and this identity satisfies $\Omega(t)=\Omega(r)^{-1}xyz\Omega(s)^{-1}$. We proceed by mathematical induction. Assume that $\Omega(t)=\Omega(r)^{-1}xyz\Omega(s)^{-1}$ for a triple $(r,t,s)\in \mathrm{F}\mathbb{T}$. It suffices to show that the statement holds for the children $(r,r\oplus t, t)$ and $(t,t\oplus s,s)$ of $(r,t,s)$, that is,
\[\Omega(r\oplus t)=\Omega(r)^{-1}xyz\Omega(t)^{-1},\quad \Omega(t\oplus s)=\Omega(t)^{-1}xyz\Omega(s)^{-1}.\]
We show the former equality. By definition of $\mathrm{W}\mathbb{T}$, we have $\Omega(r\oplus t)=\Omega(t)\Omega(s)\Omega(t)^{-1}$. By the assumption of induction, we have
\[\Omega(r\oplus t)=\Omega(r)^{-1}xyz\Omega(s)^{-1}\Omega(s)\Omega(t)^{-1}=\Omega(r)^{-1}xyz\Omega(t)^{-1},\]
as desired.  The latter can be proved in the same way.
\end{proof}

\begin{proposition}\label{prop:presnake-relation}
For a Farey triple $(r,t,s)$ with $t\in (0,\infty)\cap\mathbb Q$, the following three statements hold:
 \begin{itemize}\setlength{\leftskip}{-10pt}
     \item[(1)] Assume that $r=\frac{0}{1}$ and $s\neq\frac{1}{0}$. If the symmetric decomposition of $\omega_s$ is $w\alpha w^{-1}$, then
     \[\omega_t=yzw\alpha^{-1}w^{-1}.\]
     \item[(2)] Assume that $r\neq \frac{0}{1}$ and $s= \frac{1}{0}$. If the symmetric decomposition of $\omega_r$ is $u\beta u^{-1}$, then
     \[\omega_t=u\beta^{-1}u^{-1}xy.\]
     \item[(3)] Assume that $r\neq \frac{0}{1}$ and $s\neq \frac{1}{0}$. If the symmetric decomposition of $\omega_r$ (resp. $\omega_s$) is $u\beta u^{-1}$ (resp. $w\alpha w^{-1}$), then
     \[\omega_t=u\beta^{-1}u^{-1}xyzw\alpha^{-1} w^{-1}.\]
 \end{itemize}
\end{proposition}
To prove Proposition~\ref{prop:presnake-relation}, we introduce the pre-snake graph. For a given irreducible fraction $t\in (0,\infty)\cap\mathbb Q$, we take the union of the edges of all triangles crossed by $L_t$ and call the resulting graph the \emph{pre-snake graph}. Here a triangle is crossed when the open segment $L_t^\circ$ meets its interior, and a tile means one of the right-angled triangular cells of the modified lattice. We denote by $\mathcal{PG}(t)$ the pre-snake graph associated with $t$. We note that $\mathcal{PG}(t)$ does not contain $L_t$. For example, $\mathcal{PG}(\frac{2}{5})$ is the graph formed by the black lines in Figure~\ref{fig:ex-presnakegraph}.
\begin{proof}[Proof of Proposition~\ref{prop:presnake-relation} (1) and (2)]
We first prove (1).  Under the assumptions $r=\frac{0}{1}$ and $s\neq \frac{1}{0}$, there exists $p\in \mathbb Z_{>0}$ such that $s=\frac{1}{p}$ and $t=\frac{1}{p+1}$.  The case $p=1$ is checked directly from the construction, so assume $p\geq 2$.  Consider the letters $x,y,z,x^{-1},y^{-1},z^{-1}$ assigned to the edges of $\mathcal{PG}(t)$ using $L_t$.  Then the first two letters in $\mathcal{PG}(\frac{1}{p+1})$ are $y$ and $z$; see Figure~\ref{fig:presnake-1-8}.
\begin{figure}[ht]
    \centering
    \includegraphics[scale=0.09]{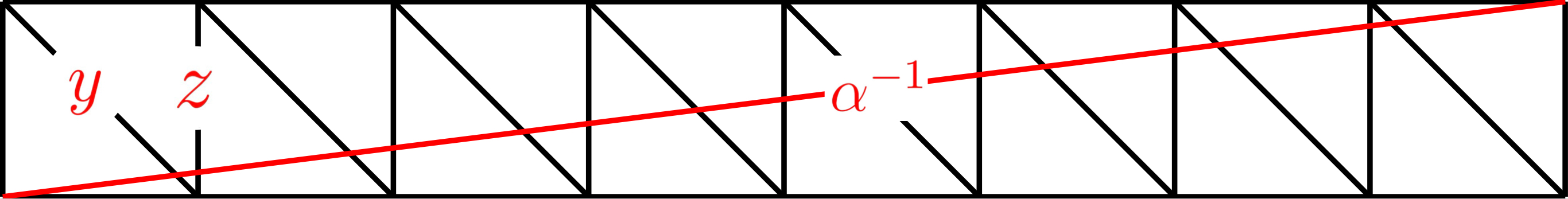}
    \caption{$\mathcal{PG}(\frac{1}{p+1})$ with $p=7$}
    \label{fig:presnake-1-8}
\end{figure}

Let us compare the letter sequence of $\mathcal{PG}(\frac{1}{p+1})$ after removing the first tile and that of $\mathcal{PG}(\frac{1}{p})$ (compare Figures~\ref{fig:presnake-1-8} and~\ref{fig:presnake-1-7}). We denote by $\mathcal{SPG}(\frac{1}{p+1})$ the former graph.
\begin{figure}[ht]
    \centering
    \includegraphics[scale=0.08]{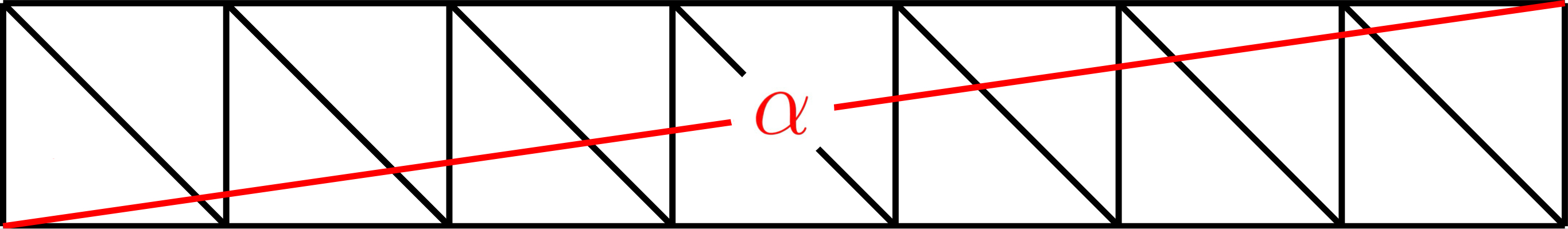}
    \caption{$\mathcal{PG}(\frac{1}{p})$ with $p=7$}
    \label{fig:presnake-1-7}
\end{figure}

We prove that only the letters associated with the central edges in $\mathcal{SPG}(\frac{1}{p+1})$ and $\mathcal{PG}(\frac{1}{p})$ differ. The central letters differ: the corresponding letter of $\mathcal{PG}(\frac{1}{p})$ is $\alpha$ (where $\alpha=y$ or $z$) by assumption, whereas the corresponding letter of $\mathcal{SPG}(\frac{1}{p+1})$ is $\alpha^{-1}$.  It remains to show that all other letters coincide.  We first consider vertical edges. The height of the intersection point between the $(a+1)$-th vertical edge from the left of $\mathcal{PG}(\frac{1}{p})$ and the line segment $L_{\frac{1}{p}}$ is $\frac{a}{p}$. Moreover, the height of the intersection point between the $(a+1)$-th vertical edge from the left of $\mathcal{SPG}(\frac{1}{p+1})$ and the line segment $L_{\frac{1}{p+1}}$ is $\frac{a+1}{p+1}$. Since $\frac{a}{p} < \frac{a+1}{p+1}$, it is sufficient to show $\frac{a+1}{p+1} \leq \frac{1}{2}$ if $\frac{a}{p} < \frac{1}{2}$. Since $2a\leq p-1$, we have
\[\frac{1}{2}-\frac{a+1}{p+1}=\frac{p-2a-1}{2(p+1)}\geq 0,\]
as desired.  The same argument applies to the letters assigned to diagonal edges.  Hence only the central letters differ, and the claim follows.

The proof of (2) is the reflected version of the proof of (1).  More precisely, apply the symmetry of the modified lattice obtained by reflecting the picture across the line of slope $1$ and reversing the orientation of the boundary strip.  Under this symmetry, horizontal and vertical edges are interchanged, diagonal edges are preserved, and the right-hand-side convention is reversed.  Hence the left boundary case $r=\frac{0}{1}$ is changed into the right boundary case $s=\frac{1}{0}$; the initial block $yz$ in (1) becomes the terminal block $xy$ in (2), while the central letter is inverted in the same way.  This gives
\[
\omega_t=u\beta^{-1}u^{-1}xy.
\]
\end{proof}

To prove Proposition~\ref{prop:presnake-relation}(3), we use two lemmas.
\begin{lemma}\label{lem:height-in-pg}
If $r=\frac{a}{b}$ and $s=\frac{c}{d}$, so that $t=\frac{a+c}{b+d}$, then $y_{b+1}-\lfloor y_{b+1} \rfloor=\frac{1}{b+d}$, where $y_i$ is the height of the intersection point of $L_t$ and the $i$-th vertical line from the left in $\mathcal{PG}(t)$.
\end{lemma}
\begin{proof}
The $(b+1)$-st vertical line from the left is the vertical line through $x=b$.  Since $L_t$ has slope $\frac{a+c}{b+d}$, its intersection with this vertical line has height
\[
y_{b+1}=\frac{b(a+c)}{b+d}=a+\frac{bc-ad}{b+d}.
\]
Because $r$ and $s$ are neighboring entries in a Farey triple with $r<s$, we have $bc-ad=1$.  Hence $y_{b+1}=a+\frac{1}{b+d}$, and the claim follows.
\end{proof}
\begin{lemma}[{\cite[Corollary~7.15]{gyoda-maruyama-sato}}]\label{lem:decomposition-lemma}
For $(r,t,s)\in \mathrm{F}\mathbb{T}$, $\mathcal{PG}(t)$ is decomposed into $\mathcal{PG}(r)$, one intervening tile, and $\mathcal{PG}(s)$ in the order from the lower left to the upper right.
\end{lemma}
\begin{proof}[Proof of Proposition~\ref{prop:presnake-relation} (3)]
By Lemma~\ref{lem:decomposition-lemma}, $\mathcal{PG}(t)$ is decomposed into $\mathcal{PG}(r)$, one intervening tile, and $\mathcal{PG}(s)$ (see Figure~\ref{fig:decomposition}). We denote by $\mathcal{SPG}(r)$ (resp. $\mathcal{SPG}(s)$) the $\mathcal{PG}(r)$- (resp. $\mathcal{PG}(s)$-)part in $\mathcal{PG}(t)$. Note that letters assigned to the tile of the second component in the decomposition are $xyz$.
\begin{figure}[ht]
    \centering
    \includegraphics[scale=0.17]{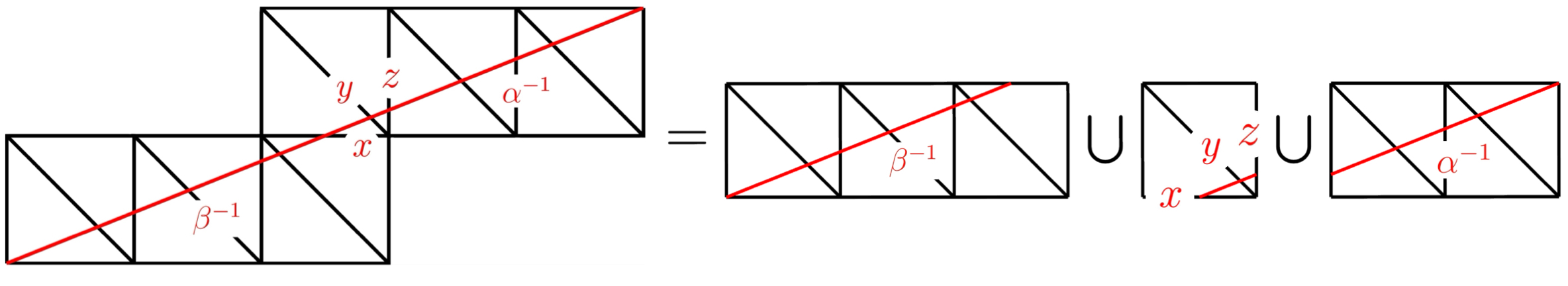}
    \caption{Decomposition of pre-snake graph with $r=\frac{1}{3}$, $t=\frac{2}{5}$, $s=\frac{1}{2}$}
    \label{fig:decomposition}
\end{figure}

Assume that $r=\frac{a}{b}$ and $s=\frac{c}{d}$.  By Lemma~\ref{lem:height-in-pg}, the vertical distance from the intersection of $L_t$ with the first vertical edge of $\mathcal{SPG}(s)$ to the lower endpoint of that edge is $\frac{1}{b+d}$.  Similarly, the horizontal distance from the intersection of $L_t$ with the rightmost horizontal edge of $\mathcal{SPG}(r)$ to the right endpoint of that edge is $\frac{1}{a+c}$. Comparing the letter sequence of $\mathcal{SPG}(r)$ and that of $\mathcal{PG}(r)$, the same argument as in case (1) shows that all letters except the central ones coincide.  The comparison between $\mathcal{SPG}(s)$ and $\mathcal{PG}(s)$ is identical.  This proves the claim.
\end{proof}

\begin{proof}[Proof of Theorem~\ref{thm:ft-wt}]
It is immediate that $\omega_{\frac{0}{1}}=x$, $\omega_{\frac{1}{1}}=y$, and $\omega_{\frac{1}{0}}=z$.  In the remaining cases, we prove that $\omega_t$ represents $\Omega(t)$.  By Lemma~\ref{prop:presnake-relation-W} and induction, it suffices to show that $\omega_t$ represents $\omega_r^{-1}xyz\omega_s^{-1}$ in $\mathfrak F_3$ for each Farey triple $(r,t,s)$.  First consider the case $r=\frac{0}{1}$.  If $\omega_s=w\alpha w^{-1}$, then Proposition~\ref{prop:presnake-relation}(1) and $\omega_r=x$ give
\[
\omega_t
=yzw\alpha^{-1}w^{-1}
=x^{-1}xyzw\alpha^{-1}w^{-1}
=\omega_r^{-1}xyz\omega_s^{-1},
\]
as desired.  In case (2), if $\omega_r=u\beta u^{-1}$, Proposition~\ref{prop:presnake-relation}(2) gives
\[
\omega_t=u\beta^{-1}u^{-1}xy
=u\beta^{-1}u^{-1}xyz z^{-1}
=\omega_r^{-1}xyz\omega_s^{-1}.
\]
In case (3), if $\omega_r=u\beta u^{-1}$ and $\omega_s=w\alpha w^{-1}$, Proposition~\ref{prop:presnake-relation}(3) gives
\[
\omega_t=u\beta^{-1}u^{-1}xyzw\alpha^{-1}w^{-1}
=\omega_r^{-1}xyz\omega_s^{-1}.
\]
Therefore, $\omega_t$ represents $\Omega(t)$ for any $t\in\mathbb Q_{\ge0}^{\mathrm{ext}}$.  Thus the map $\Phi$ sends the root of $\mathrm{F}\mathbb T$ to the root of $\mathrm{W}\mathbb T$ and preserves left and right children.  Hence $\Phi$ is the isomorphism of rooted planar binary trees from $\mathrm{F}\mathbb T$ to $\mathrm{W}\mathbb T$.
\end{proof}

\subsection{Markov-monodromy matrices}\label{sec:mm-matrices}
We now evaluate the words constructed above as matrices.  This is the point where the parameters $(k_1,k_2,k_3)$ and the label permutation $\sigma$ enter the construction.  Fix $(k_1,k_2,k_3)\in\mathbb Z_{\ge0}^3$ and $\sigma\in\mathfrak S_3$, and set $K=3+k_1+k_2+k_3$.  Assign to the letters $x,y,z$ the following matrices in $\mathrm{SL}(2,\mathbb Z)$:
\begin{align*}
X=\begin{bmatrix}-k_{\sigma(1)}&-1\\1&0\end{bmatrix},\quad
Y=\begin{bmatrix}1&-1\\ k_{\sigma(2)}+2&-k_{\sigma(2)}-1\end{bmatrix},\quad
Z=\begin{bmatrix}1&-k_{\sigma(3)}-2\\1&-k_{\sigma(3)}-1\end{bmatrix}.
\end{align*}
For an irreducible fraction $t\in\mathbb Q_{\ge0}^{\mathrm{ext}}$, define
\[
M_t:=\omega_t\big|_{x\mapsto X,\;y\mapsto Y,\;z\mapsto Z}.
\]
We call $M_t$ the \emph{Markov-monodromy matrix} attached to $t$ for
the fixed data $(k_1,k_2,k_3,\sigma)$.  We use the normalization in which
the generalized Markov number appears in the $(2,1)$-entry and the
characteristic number appears in the $(1,1)$-entry.  This differs from the
normalization used in \cite{gyoda-maruyama-sato} and \cite{bana-gyo}, where the generalized
Markov number is placed in the $(1,2)$-entry.  The present convention is
chosen so as to match the word construction and the continued-fraction
formula below.

Under this substitution, the word-tree recursion is transported to triples
of matrices.  If $(a,b,c)$ is a vertex of the word tree and $(A,B,C)$ is
the corresponding substituted matrix triple, then the two child operations
become
\[
(A,B,C)\mapsto(A,BCB^{-1},B),
\qquad
(A,B,C)\mapsto(B,B^{-1}AB,C).
\]
The root triple has lower-left entries $
  (1,k_{\sigma(2)}+2,1)$.
Moreover, under the above two operations, these lower-left entries transform
by the same exchange relations as those in the
$(k_1,k_2,k_3,\sigma)$-GM tree.  The position labels are transported along
the tree according to $\sigma$.  Thus the matrix recursion realizes the
same labeled GM tree.

The resulting entry formula is as follows.  In the equal-parameter case
$k_1=k_2=k_3=k$, this is the $k$-Markov-monodromy entry formula of
\cite[Theorems~7.10, 7.27(1)]{gyoda-maruyama-sato}, rewritten in the
present $(2,1)$-entry convention; the three-parameter version is treated in
\cite[Theorem~10.1]{bana-gyo}.
\begin{theorem}\label{thm:mm-matrix-entries}
Let $(k_1,k_2,k_3)\in \mathbb Z_{\geq 0}^3$ and
$\sigma\in \mathfrak S_3$.  For $t\in (0,\infty)\cap\mathbb Q$, assume that
\[
  s^{\circ}_{k_1,k_2,k_3,\sigma}(t)=(a_1,\dots,a_n).
\]
Let $m_t$ be the corresponding $(k_1,k_2,k_3)$-GM number, let $u_t$
be its characteristic number, and let $k_t=k_{i_t}$ be the coefficient
selected by the position label of $m_t$.  Then
\[
M_t=
\begin{bmatrix}
    u_t&-\dfrac{u_t^2+k_tu_t+1}{m_t}\\[6pt]
    m_t&-(u_t+k_t)
\end{bmatrix}
=
\begin{bmatrix}
 N(a_2,\dots,a_n)&-N(a_2,\dots,a_{n-1})\\
 N(a_1,\dots,a_n)&-N(a_1,\dots,a_{n-1})
\end{bmatrix}.
\]
\end{theorem}

\begin{example}\label{ex:mm-matrix-entries-two-fifths}
We continue Example~\ref{ex:theorem-gm-sequence-two-fifths}.  Take
$(k_1,k_2,k_3)=(1,2,0)$, $\sigma=\mathrm{id}$, and
$t=\frac{2}{5}$.  Then $K=6$, and the matrices assigned to the letters are
\[
X=\begin{bmatrix}-1&-1\\1&0\end{bmatrix},\quad
Y=\begin{bmatrix}1&-1\\4&-3\end{bmatrix},\quad
Z=\begin{bmatrix}1&-2\\1&-1\end{bmatrix}.
\]
Using the word
\[
\omega_{\frac{2}{5}}=yzy^{-1}z^{-1}y^{-1}xyzyz^{-1}y^{-1}
\]
from Example~\ref{ex:continued-fraction-Markov}, substitution gives
\[
\begin{aligned}
M_{\frac{2}{5}}
&=YZY^{-1}Z^{-1}Y^{-1}XYZYZ^{-1}Y^{-1}\\
&=\begin{bmatrix}1930&-453\\8227&-1931\end{bmatrix}.
\end{aligned}
\]
The denominator of the continued fraction $[4;3,1,4,5,1,3,4]$ also
explains the characteristic number.  In the equal-parameter setting, this
denominator interpretation follows from
\cite[Theorem~7.25]{gyoda-maruyama-sato}.  The same argument applies to the
present three-parameter setting, although we do not reproduce the proof here.
In the present notation, Theorem~\ref{thm:mm-matrix-entries} gives
\[
 N\left(3,1,4,5,1,3,4\right)
 =\operatorname{den}[4;3,1,4,5,1,3,4]
 =1930,
\]
so this computation gives the characteristic number.
On the other hand, Example~\ref{ex:theorem-gm-sequence-two-fifths} gives
$m_{\frac{2}{5}}=8227$, and the position label of $m_{\frac{2}{5}}$ is $1$, so
$k_{\frac{2}{5}}=k_1=1$.  Hence Theorem~\ref{thm:mm-matrix-entries}
reads
\[
M_{\frac{2}{5}}=
\begin{bmatrix}
u_{\frac{2}{5}}&-\dfrac{u_{\frac{2}{5}}^2+u_{\frac{2}{5}}+1}{8227}\\[6pt]
8227&-(u_{\frac{2}{5}}+1)
\end{bmatrix}.
\]
Comparing the entries gives
\[
u_{\frac{2}{5}}=1930,\qquad
\frac{1930^2+1930+1}{8227}=453.
\]
Thus the $(2,1)$-entry is the GM number $8227$, while the
$(1,1)$-entry is the characteristic number $1930$, as asserted by
Theorem~\ref{thm:mm-matrix-entries}.
\end{example}

\section{\texorpdfstring{Obtaining Cohn words from $\omega_t$}{Obtaining Cohn words from omega t}}\label{sec:mm-cohn}
This section specializes the word comparison to the classical Cohn-word side. 
\subsection{Cohn words and Cohn tree}
Let $\mathfrak M_3$ be the free (non-commutative) monoid of rank $3$ generated by $\{p,q,r\}$.
We consider the following tree, whose vertices are $\mathfrak M_3^3$:
\begin{itemize}\setlength{\leftskip}{-10pt}
\item[(1)] The root vertex is $(p,q,r)\in \mathfrak M_3^3$.
\item[(2)] For a vertex $(a,b,c)$, its two children are defined as
\[\begin{xy}(0,0)*+{(a,b,c)}="1",(-20,-10)*+{(a,ab,b)}="2",(20,-10)*+{(b,bc,c).}="3", \ar@{-}"1";"2"\ar@{-}"1";"3"
\end{xy}\]
\end{itemize}
This tree is called the \emph{Cohn word tree}, and we denote it by $\mathrm{CoW}\mathbb T$.
We call an element of $\mathfrak M_3$ appearing in $\mathrm{CoW}\mathbb T$ a \emph{Cohn word}.

The first few Cohn words in $\mathrm{CoW}\mathbb T$ are as follows:
\begin{align*}
\begin{xy}(0,0)*+{(p,q,r)}="0",(20,20)*+{\left(q,qr,r\right)}="1",(20,-20)*+{\left(p,pq,q\right)}="1'",(40,30)*+{\left(qr,qr^2,r\right)}="20",(40,10)*+{\left(q,q^2r,qr\right)}="21",(40,-10)*+{\left(pq,pq^2,q\right)}="22",(40,-30)*+{\left(p,p^2q,pq\right)}="23",(80,35)*+{\left(qr^2,qr^3,r\right)}="30",(80,25)*+{\left(qr,qrqr^2,qr^2\right)}="31",(80,15)*+{\left(q^2r,q^2rqr,qr\right)}="32",(80,5)*+{\left(q,q^3r,q^2r\right)}="33",(80,-5)*+{\left(pq^2,pq^3,q\right)}="34",(80,-15)*+{\left(pq,pqpq^2,pq^2\right)}="35",(80,-25)*+{\left(p^2q,p^2qpq,pq\right)}="36",(80,-35)*+{\left(p,p^3q,p^2q\right)}="37",\ar@{-}"0";"1"\ar@{-}"0";"1'"\ar@{-}"1";"20"\ar@{-}"1";"21"\ar@{-}"1'";"22"\ar@{-}"1'";"23"\ar@{-}"20";"30"\ar@{-}"20";"31"\ar@{-}"21";"32"\ar@{-}"21";"33"\ar@{-}"22";"34"\ar@{-}"22";"35"\ar@{-}"23";"36"\ar@{-}"23";"37"
\end{xy}.
\end{align*}

We now recall the geometric realization of Cohn words.
\begin{remark}
Usually, Cohn words are defined using two letters. In this text, however, we use three letters, $p,q,r$, for compatibility with other properties. The Cohn word tree used in texts such as \cite{aig} corresponds to the full subtree of the Cohn word tree in this paper rooted at $(p,pq,q)$.
\end{remark}
For $t\in(0,\infty)\cap\mathbb Q$, let $c(t)$ denote the element of $\mathfrak M_3$ in $\mathrm{CoW}\mathbb T$ located at the unique vertex $(r,t,s)$ of the Farey tree whose middle component is $t$.  For the boundary cases, set $c(\frac{0}{1})=p$ and $c(\frac{1}{0})=r$.
We recall a geometric construction of $c(t)$.

Set $c_{\frac{0}{1}}:=p$ and $c_{\frac{1}{0}}:=r$.  For an irreducible fraction $t=\frac{a}{b}\in (0,\infty)\cap\mathbb Q$, we define a word $c_t$ as follows.
Let $L_t$ be the line segment in $\mathbb{R}^2$ from $(0,0)$ to $(b,a)$, oriented from $(0,0)$ to $(b,a)$.  Each time $L_t$ intersects a line of the integer lattice, take the lattice point adjacent to that intersection on the right-hand side of the oriented line.  If the intersection is itself a lattice point, we use the convention that this lattice point is recorded; repeated lattice points will be discarded below.  Record the resulting lattice points in order as
\[
 P_1,P_2,\dots,P_n.
\]
We set $P_1=(0,0)$ and $P_n=(b,a)$, and we discard repeated lattice points if they occur.  Connect the points $P_1,P_2,\dots,P_n$ in order by line segments.  To each segment we assign a letter according to its slope:
\[
 0\mapsto p,\qquad 1\mapsto q,\qquad \infty\mapsto r.
\]
The sequence of these letters, read in order, is defined to be $c_t$.

\begin{remark}
The word $c_t$ is a Christoffel word in the classical terminology.
\end{remark}

\begin{example}
The word $c_{\frac{2}{5}}$ is $p^2qpq$, and $c_{\frac{5}{2}}$ is $qrqr^2$. See Figure~\ref{fig:christoffel}.
\begin{figure}[ht]
    \centering
    \begin{tikzpicture}
    \draw[thick] (0,0) grid (5,2);
    \draw[red,thick] (0,0) -- (5,2);
    \draw[line width=2pt] (0,0) -- (2,0)-- (3,1) -- (4,1) -- (5,2);
    \node at (0.5,-0.5) {$p$};
    \node at (1.5,-0.5) {$p$};
    \node at (2.5,-0.5) {$q$};
    \node at (3.5,-0.5) {$p$};
    \node at (4.5,-0.5) {$q$};
    \node[red] at (1.5,0.6) {\rotatebox{20}{$>$}};
    \node[red] at (3.5,1.4) {\rotatebox{20}{$>$}};
\end{tikzpicture}
\hspace{1cm}
 \begin{tikzpicture}
    \draw[thick] (0,0) grid (2,5);
    \draw[red,thick] (0,0) -- (2,5);
    \draw[line width=2pt] (0,0) -- (1,1)-- (1,2) -- (2,3) -- (2,5);
    \node at (2.5,0.5) {$q$};
    \node at (2.5,1.5) {$r$};
    \node at (2.5,2.5) {$q$};
    \node at (2.5,3.5) {$r$};
    \node at (2.5,4.5) {$r$};
    \node[red] at (0.6,1.5) {\rotatebox{65}{$>$}};
    \node[red] at (1.4,3.5) {\rotatebox{65}{$>$}};
\end{tikzpicture}
    \caption{The words $c_{\frac{2}{5}}$ and $c_{\frac{5}{2}}$}
    \label{fig:christoffel}
\end{figure}
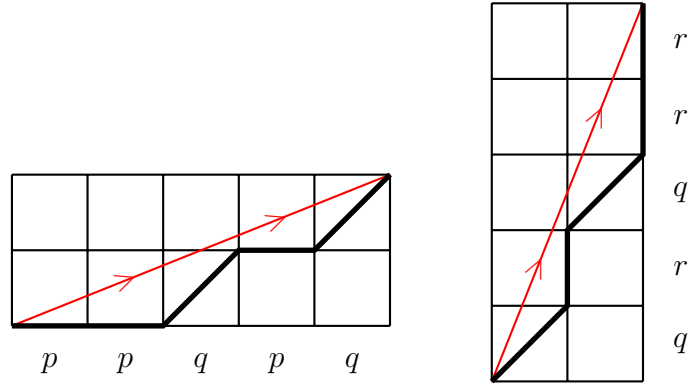
\end{example}

We use the following standard description of Cohn words.

\begin{theorem}[{\cite[Theorem~7.6]{aig}}]\label{thm:cohn-geometric}
For any irreducible fraction $t\in(0,\infty)\cap\mathbb Q$, we have $c_t=c(t)$.  The boundary cases $t=\frac{0}{1}$ and $t=\frac{1}{0}$ hold by the defining conventions $c_{\frac{0}{1}}=c(\frac{0}{1})=p$ and $c_{\frac{1}{0}}=c(\frac{1}{0})=r$.
\end{theorem}

We shall use the following consequence.

\begin{theorem}[{\cite{aig}}]
For any irreducible fraction $t\in ((0,1)\cup(1,\infty))\cap\mathbb Q$, there exists a palindrome word $w$ such that
\begin{itemize}\setlength{\leftskip}{-10pt}
    \item[(1)] $c_t=pwq$ if $t\in (0,1)$,
    \item[(2)] $c_t=qwr$ if $t\in (1,\infty)$.
\end{itemize}
\end{theorem}
We now compare $\omega_t$ with the corresponding Cohn word.  The next lemma records the elementary local pattern of edge crossings that will be used in the comparison.

\begin{lemma}\label{lem:cohn-subword-pattern}
Let $t\in(0,\infty)\cap\mathbb Q$ be a positive reduced fraction.
\begin{itemize}\setlength{\leftskip}{-10pt}
\item[(1)] If $0<t<1$, then each subword of $\omega_t$ lying between two consecutive occurrences of $z^{\pm1}$ is either of the form $y^{\pm1}$ or of the form $y^{\pm1}x^{\pm1}y^{\pm1}$.
\item[(2)] If $1<t<\infty$, then each subword of $\omega_t$ lying between two consecutive occurrences of $x^{\pm1}$ is either of the form $y^{\pm1}$ or of the form $y^{\pm1}z^{\pm1}y^{\pm1}$.
\end{itemize}
\end{lemma}

\begin{proof}
Assume first that $0<t<1$.  Between two successive crossings of vertical lattice edges, the segment lies in one vertical strip.  In such a strip it crosses a diagonal edge once if it stays in the same horizontal row, and it crosses two diagonal edges with one horizontal edge between them if it moves to the next row.  Translating these three edge types into the letters $x,y,z$ gives the two forms in (1).  The proof of (2) is the reflected argument: use horizontal strips instead of vertical strips, so that the roles of horizontal and vertical edges are interchanged while diagonal edges remain diagonal.
\end{proof}

The next result shows that the palindrome part of the Cohn word is obtained directly from $\omega_t$.

\begin{theorem}\label{thm:mm-cohn}
 For the word $\omega_t$ attached to an irreducible fraction $t\in((0,1)\cup(1,\infty))\cap\mathbb Q$, define a word $w$ in the alphabet $\{p,q,r\}$ by the following rule.
\begin{description}
  \item[\textnormal{Case $0 < t < 1$}]
  For each subword of $\omega_t$ lying between two consecutive occurrences of $z^{\pm1}$, assign
  \[
  p \text{ if the subword is } y^{\pm1}, \quad
  q \text{ if the subword is } y^{\pm1}x^{\pm1}y^{\pm1}.
  \]
  The word $w$ is obtained by concatenating these assigned letters in the order in which they appear in $\omega_t$.

  \item[\textnormal{Case $1 < t < \infty$}]
  For each subword of $\omega_t$ lying between two consecutive occurrences of $x^{\pm1}$, assign
  \[ q \text{ if the subword is } y^{\pm1}z^{\pm1}y^{\pm1},
  \quad r \text{ if the subword is } y^{\pm1}.
  \]
  The word $w$ is again obtained by concatenating the assigned letters in order.
\end{description}

Then the Cohn word $c_t$ is given by
\[c_t=
\begin{cases}
pwq, & \text{if } 0 < t < 1, \\[4pt]
qwr, & \text{if } 1 < t < \infty.
\end{cases}
\]
\end{theorem}

\begin{proof}
The letters in $\omega_t$ record, in order, the three types of edges crossed by $L_t$.  By Lemma~\ref{lem:cohn-subword-pattern}, the subwords specified in the statement are exactly the local pieces occurring between consecutive boundary letters.  These local pieces correspond to the letters in the middle part of the classical Cohn word: for $0<t<1$ the pieces between vertical boundary letters give $p$ or $q$, and for $1<t<\infty$ the pieces between horizontal boundary letters give $q$ or $r$.  The endpoint letters give the outer factors $p,q$ or $q,r$.  The geometric description of Cohn words in Theorem~\ref{thm:cohn-geometric} then identifies the resulting word with $c_t$.
\end{proof}

\begin{example}\label{ex:mm-cohn-two-fifths}
We again take $t=\frac{2}{5}$.  From Example~\ref{ex:continued-fraction-Markov},
\[
\omega_{\frac{2}{5}}=yzy^{-1}z^{-1}y^{-1}xyzyz^{-1}y^{-1}.
\]
Since $0<\frac{2}{5}<1$, we look at the subwords between consecutive occurrences of
$z^{\pm1}$.  These subwords are $y^{-1}, y^{-1}xy, y$.
By the rule in Theorem~\ref{thm:mm-cohn}, they correspond respectively to $
p, q, p$.
Thus $w=pqp$, and hence
\[
c_{\frac{2}{5}}=pwq=p^2qpq.
\]
This agrees with the direct geometric computation of the Cohn word in
Figure~\ref{fig:christoffel}.
\end{example}

\begin{remark}
The remaining central slope and the boundary cases are handled by the defining conventions: $\omega_{\frac{1}{1}}=y$ corresponds to $q=c_{\frac{1}{1}}$, while $\omega_{\frac{0}{1}}=x$ and $\omega_{\frac{1}{0}}=z$ correspond to $p=c_{\frac{0}{1}}$ and $r=c_{\frac{1}{0}}$, respectively.
\end{remark}

\section{Generalized Cohn matrices and endpoint-completed words}\label{sec:gc-extended}

In this section we put the preceding word construction together with the generalized Cohn matrix construction.  The point is that the shifted segment $\overline L_t$ carries two compatible readings.  The sign reading gives a generalized strongly admissible sequence and hence a generalized Cohn matrix.  The edge-letter reading gives the endpoint-completed word $\overline{\omega}_t$.  We explain both readings and then spell out how the generalized Cohn matrix is obtained from the endpoint-completed word.

\subsection{Generalized Cohn matrices}\label{subsec:gc-matrices}
Fix $(k_1,k_2,k_3)\in\mathbb Z_{\ge0}^3$ and $\sigma\in\mathfrak S_3$, and set
\[
K:=3+k_1+k_2+k_3.
\]
The initial generalized Cohn matrices are
\begin{align}
C_{\frac{0}{1}}&=\begin{bmatrix}K&-Kk_{\sigma(1)}-1\\1&-k_{\sigma(1)}\end{bmatrix},\nonumber\\
C_{\frac{1}{1}}&=\begin{bmatrix}K(k_{\sigma(2)}+2)-k_{\sigma(2)}-1&K-1\\k_{\sigma(2)}+2&1\end{bmatrix},\label{eq:gc-initial}\\
C_{\frac{1}{0}}&=\begin{bmatrix}K-k_{\sigma(3)}-1&K-k_{\sigma(3)}-2\\1&1\end{bmatrix}.\nonumber
\end{align}
For a Farey triple $(r,t,s)$, write the corresponding labeled vertex as
\[
((m_r,i_r),(m_t,i_t),(m_s,i_s)).
\]
Put $k_r:=k_{i_r}$, and define
\[
D_r:=\begin{bmatrix}k_r&Kk_r\\0&k_r\end{bmatrix}.
\]
The generalized Cohn matrices are then defined recursively by
\begin{equation}\label{eq:gc-recursion}
C_{r\oplus t}=C_rC_t-D_s,
\qquad
C_{t\oplus s}=C_tC_s-D_r.
\end{equation}
We call $C_t$ the \emph{$(k_1,k_2,k_3,\sigma)$-generalized Cohn matrix}, or simply the \emph{GC matrix}, attached to $t$.

The entries of $C_t$ are described by the corresponding generalized Markov number and characteristic number.  We record this as a theorem.  In the equal-parameter case $k_1=k_2=k_3=k$, this is the generalized Cohn matrix formula of \cite[Theorem~7.27(2)]{gyoda-maruyama-sato}, rewritten in the present lower-row convention; the three-parameter matrixization is treated in \cite[Theorem~9.3]{bana-gyo}.
\begin{theorem}\label{thm:gc-explicit}
Let $t\in(0,\infty)\cap\mathbb Q$.  Let $m_t$ be the GM number at $t$, let $u_t$ be its characteristic number, and let $k_t=k_{i_t}$ be the coefficient selected by the position label of $m_t$.  Then
\begin{equation}\label{eq:gc-explicit}
C_t=
 \begin{bmatrix}
 Km_t-k_t-u_t & \displaystyle\frac{Km_tu_t-k_tu_t-u_t^2-1}{m_t}\\[6pt]
 m_t&u_t
 \end{bmatrix}.
\end{equation}
\end{theorem}
We use this lower-row placement of $m_t$ and $u_t$ as the Cohn-compatible convention for the generalized Cohn matrix $C_t$.

\subsection{\texorpdfstring{Strongly admissible sequences and generalized Cohn matrices from $\overline L_t$}{Strongly admissible sequences and generalized Cohn matrices from overline Lt}}\label{subsec:sequence-to-gc}
Let $t=\frac{p}{q}\in(0,\infty)\cap\mathbb Q$ be a positive reduced fraction.  In $\widetilde{\mathbb{R}^{2}}$ take $A=(0,0)$ and $B=(q,p)$, and let $L_t$ be the oriented segment from $A$ to $B$.  Choose $\varepsilon>0$ sufficiently small so that the shifted segment has the same order of edge and triangle events away from the endpoints and does not pass through any lattice vertex or any midpoint of an edge.  Translate the whole segment horizontally by the vector $(-\varepsilon,0)$.  The resulting segment, from $(-\varepsilon,0)$ to $(q-\varepsilon,p)$, is denoted by $\overline L_t$.  We call it the endpoint-completed segment associated with $t$; see Figure~\ref{fig:overlineLt} for the case $t=\frac{2}{5}$.

The sequence attached to $\overline L_t$ is obtained by applying the sign-reading rules of Subsection~\ref{subsec:gm-sequences-from-line-segments} to this shifted segment.  Thus all triangle and edge contributions along the interior of $\overline L_t$ are read exactly as for $L_t$.  We only modify the treatment of the two endpoints.  Orient $\overline L_t$ from its lower-left endpoint to its upper-right endpoint.  The first lower-left edge met at the initial endpoint is counted as crossed, whereas the last upper-right edge containing the terminal endpoint is not counted as crossed.  After this half-open endpoint convention is imposed, the resulting signs are recorded in their order along $\overline L_t$.

\begin{figure}[ht]
    \centering
    \begin{tikzpicture}[x=1.2cm,y=1.2cm]
    \definecolor{segmentred}{RGB}{255,0,0}
    \tikzset{
      figgrid/.style={black,line width=0.7pt,line cap=butt,line join=miter},
      figsegment/.style={draw=segmentred,line width=1.0pt,line cap=round,line join=round}
    }
    \draw[figgrid] (0,0) -- (3,0);
    \draw[figgrid] (0,1) -- (5,1);
    \draw[figgrid] (2,2) -- (5,2);
    \draw[figgrid] (0,0) -- (0,1);
    \draw[figgrid] (1,0) -- (1,1);
    \draw[figgrid] (2,0) -- (2,2);
    \draw[figgrid] (3,0) -- (3,2);
    \draw[figgrid] (4,1) -- (4,2);
    \draw[figgrid] (5,1) -- (5,2);
    \draw[figgrid] (0,1) -- (1,0);
    \draw[figgrid] (1,1) -- (2,0);
    \draw[figgrid] (2,1) -- (3,0);
    \draw[figgrid] (2,2) -- (3,1);
    \draw[figgrid] (3,2) -- (4,1);
    \draw[figgrid] (4,2) -- (5,1);
    \draw[figsegment] (-0.18,0) -- (4.82,2);
    \end{tikzpicture}
    \caption{The shifted segment $\overline L_t$ associated with $t=\frac{2}{5}$}
    \label{fig:overlineLt}
\end{figure}
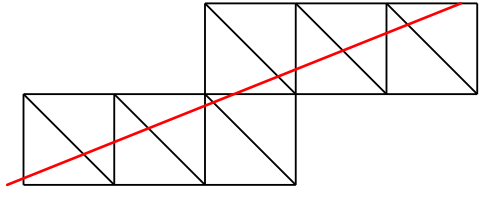

We define the \emph{$(k_1,k_2,k_3,\sigma)$-generalized strongly admissible sequence} attached to $t$ by compressing this signed list.  Namely, if the signs obtained from the above reading form maximal consecutive blocks of equal signs of lengths $
a_0,a_1,\dots,a_\ell$,
then we set
\[
s_{k_1,k_2,k_3,\sigma}(t):=(a_0,a_1,\dots,a_\ell).
\]
The resulting sequence $s_{k_1,k_2,k_3,\sigma}(t)$ is the generalized strongly admissible sequence associated with $\overline L_t$.

For example, when $(k_1,k_2,k_3)=(1,2,0)$, $\sigma=\mathrm{id}$, and $t=\frac{2}{5}$, the shifted segment carries the signs shown in Figure~\ref{fig:overlineLt-signed}.  Reading the maximal runs of equal signs gives
\[
 s_{k_1,k_2,k_3,\sigma}\left(\frac{2}{5}\right)=(5,1,3,3,1,5,4,1,3,4).
\]

\begin{figure}[ht]
\centering
\begin{tikzpicture}[x=0.01cm,y=0.01cm,scale=0.5]
\definecolor{annblue}{RGB}{255,0,0}
\tikzset{
  gridline/.style={black,line width=0.7pt,line cap=butt,line join=miter},
  gammaL/.style={draw=annblue,line width=1.1pt,line cap=round,line join=round},
  midpoint/.style={circle,fill=gray,inner sep=0pt,minimum size=1pt},
  enddot/.style={circle,fill=black,inner sep=0pt,minimum size=1.8pt},
  triwhite/.style={circle,fill=white,draw=none,inner sep=0.45pt,minimum size=8.0pt,outer sep=0pt},
  edgewhite/.style={circle,fill=white,draw=none,inner sep=0.35pt,minimum size=7.6pt,outer sep=0pt},
  trisign/.style={text=annblue,font=\scriptsize\bfseries,inner sep=0.45pt,minimum size=5.2pt,outer sep=0pt},
  edgesign/.style={text=annblue,font=\scriptsize\bfseries,inner sep=0.35pt,minimum size=4.9pt,outer sep=0pt}
}
\begin{scope}[shift={(300,170)}, x={(260,0)}, y={(0,260)}]
\begin{scope}
  \clip (-1,0) rectangle (5,2);
  \foreach \i in {-1,...,4}{
    \foreach \j in {0,1}{
      \draw[gridline] (\i,{\j+1}) -- ({\i+1},\j);
    }
  }
  \foreach \i in {-1,...,5}{
    \draw[gridline] (\i,0) -- (\i,2);
  }
  \foreach \j in {0,1,2}{
    \draw[gridline] (-1,\j) -- (5,\j);
  }
  \foreach \i in {-1,...,4}{
    \foreach \j in {0,1,2}{
      \node[midpoint] at ({\i+0.5},\j) {};
    }
  }
  \foreach \i in {-1,...,5}{
    \foreach \j in {0,1}{
      \node[midpoint] at (\i,{\j+0.5}) {};
    }
  }
  \foreach \i in {-1,...,4}{
    \foreach \j in {0,1}{
      \node[midpoint] at ({\i+0.5},{\j+0.5}) {};
    }
  }
\end{scope}
\coordinate (A) at (0,0);
\coordinate (B) at (5,2);
\def\xleftshift{0.18}
\coordinate (ALshift) at ({-\xleftshift},0);
\coordinate (BLshift) at ({5-\xleftshift},2);
\draw[gammaL] (ALshift) -- (BLshift);
\foreach \x/\y/\sig in {-0.5/0.18/-,-0.18/0.3/-,0.20/0.22/+,0.78/0.78/-,1.20/0.22/+,1.72/0.86/-,2.18/0.55/+,2.70/0.70/+,2.30/1.18/-,2.78/1.4/-,3.15/1.03/+,3.82/1.80/-,4.18/1.6/+,4.68/1.78/+}{%
  \node[triwhite] at (\x,\y) {};
}
\foreach \x/\y/\sig in {-0.3/-0.1/-,-0.58/0.40/-,-0.36/0.40/-,0.31/0.48/-,0.54/0.48/-,1.36/0.45/+,1.60/0.45/+,2.36/0.54/+,2.60/0.54/+,2.58/1/+,2.32/1.47/-,2.55/1.47/-,3.46/1.4/-,3.70/1.4/-,4.41/1.6/+,4.65/1.6/+}{%
  \node[edgewhite] at (\x,\y) {};
}
\foreach \x/\y/\sig in {-0.5/0.18/-,-0.18/0.3/-,0.20/0.22/+,0.78/0.78/-,1.20/0.22/+,1.72/0.86/-,2.18/0.55/+,2.70/0.70/+,2.30/1.18/-,2.78/1.4/-,3.15/1.03/+,3.82/1.80/-,4.18/1.6/+,4.68/1.78/+}{%
  \node[trisign] at (\x,\y) {$\sig$};
}
\foreach \x/\y/\sig in {-0.3/-0.1/-,-0.58/0.40/-,-0.36/0.40/-,0.31/0.48/-,0.54/0.48/-,1.36/0.45/+,1.60/0.45/+,2.36/0.54/+,2.60/0.54/+,2.58/1/+,2.32/1.47/-,2.55/1.47/-,3.46/1.4/-,3.70/1.4/-,4.41/1.6/+,4.65/1.6/+}{%
  \node[edgesign] at (\x,\y) {$\sig$};
}
\node[enddot] at (A) {};
\node[enddot] at (B) {};
\end{scope}
\end{tikzpicture}
\caption{The shifted segment $\overline L_t$ for $t=\frac{2}{5}$ with the signs assigned by the triangle-crossing and edge-crossing rules.}
\label{fig:overlineLt-signed}
\end{figure}

We now recall how the corresponding matrix is obtained from this run-length sequence.  For a finite sequence $S=(a_0,a_1,\dots,a_\ell)$ of positive integers, set
\begin{equation}\label{eq:continued-fraction-matrix-F}
F_S:=
\begin{bmatrix}a_0&1\\1&0\end{bmatrix}
\begin{bmatrix}a_1&1\\1&0\end{bmatrix}
\cdots
\begin{bmatrix}a_\ell&1\\1&0\end{bmatrix}.
\end{equation}
The generalized Cohn matrix is obtained by applying this continued-fraction matrix product to the generalized strongly admissible sequence.
\begin{theorem}[{\cite{bana-gyo,gyoda2025}}]\label{thm:gc-from-strongly-admissible}
For every positive reduced fraction $t\in(0,\infty)\cap\mathbb Q$, one has
\begin{equation}\label{eq:gc-cf}
C_t=F_{s_{k_1,k_2,k_3,\sigma}(t)}.
\end{equation}
Equivalently, if $s_{k_1,k_2,k_3,\sigma}(t)=(a_0,a_1,\dots,a_\ell)$, then
\begin{equation}\label{eq:gc-cf-N}
C_t=
\begin{bmatrix}
N(a_0,\dots,a_\ell)&N(a_0,\dots,a_{\ell-1})\\
N(a_1,\dots,a_\ell)&N(a_1,\dots,a_{\ell-1})
\end{bmatrix}.
\end{equation}
\end{theorem}

\begin{example}\label{ex:gc-cf-two-fifths}
Let $(k_1,k_2,k_3)=(1,2,0)$, $\sigma=\mathrm{id}$, and $t=\frac{2}{5}$.
The generalized strongly admissible sequence computed from
$\overline L_{\frac{2}{5}}$ is
\[
s_{1,2,0,\mathrm{id}}\left(\frac{2}{5}\right)
=(5,1,3,3,1,5,4,1,3,4).
\]
Therefore Theorem~\ref{thm:gc-from-strongly-admissible} gives
\[
C_{\frac{2}{5}}
=F_{(5,1,3,3,1,5,4,1,3,4)}
=\begin{bmatrix}47431&11127\\8227&1930\end{bmatrix}.
\]
Thus the $(2,1)$-entry is the GM number $m_{\frac{2}{5}}=8227$ from
Example~\ref{ex:theorem-gm-sequence-two-fifths}, and the $(2,2)$-entry is
the characteristic number $u_{\frac{2}{5}}=1930$ found in
Example~\ref{ex:mm-matrix-entries-two-fifths}.
\end{example}

Thus the construction of $C_t$ from $\overline L_t$ is completely explicit: write down the signs along $\overline L_t$, pass to the run-length sequence $s_{k_1,k_2,k_3,\sigma}(t)$, and multiply the elementary continued-fraction matrices in \eqref{eq:continued-fraction-matrix-F}.  For $t\in(0,\infty)\cap\mathbb Q$, comparing \eqref{eq:gc-cf} with \eqref{eq:gc-explicit} shows that the GM number $m_t$ appears as the $(2,1)$-entry and the characteristic number $u_t$ appears as the $(2,2)$-entry of $F_{s_{k_1,k_2,k_3,\sigma}(t)}$.

\subsection{Endpoint-completed words and generalized Cohn matrices}\label{subsec:extended-mm-word}
We next extend the construction of $\omega_t$ from $L_t$ to $\overline L_t$ and explain how to obtain the generalized Cohn matrix directly from the resulting word.  Read the edges crossed by $\overline L_t$ in order from the lower-left endpoint to the upper-right endpoint, using the same half-open endpoint convention described above: the first edge is counted, while the final edge containing the terminal endpoint is not.  Assign letters by the same edge-letter rule as in Section~\ref{sec:mm-words}:
\begin{itemize}
\item a horizontal edge contributes $x$ or $x^{-1}$,
\item a diagonal edge contributes $y$ or $y^{-1}$,
\item a vertical edge contributes $z$ or $z^{-1}$,
\end{itemize}
where the exponent is determined, as before, by whether the midpoint of the edge lies on the right-hand side of the oriented segment; the chosen shift ensures that no edge midpoint lies on the segment.  The word obtained in this way is denoted by $\overline{\omega}_t$.

The following lemma records the relation between the original word and the endpoint-completed word.

\begin{lemma}\label{lem:endpoint-completed-word}
For every positive reduced fraction $t\in(0,\infty)\cap\mathbb Q$, one has
\begin{equation}\label{eq:extended-mm-word}
\overline{\omega}_t=xyz\omega_t^{-1}.
\end{equation}
\end{lemma}

\begin{proof}
By the choice of the endpoint-completed segment, the only new initial
crossings are the three endpoint crossings of types horizontal, diagonal, and
vertical.  These give the initial factor $xyz$.  It remains to identify the
word contributed by the remaining, non-endpoint crossings.

Use the symmetric decomposition of Proposition~\ref{prop:symmetric-decomposition}
and write
\[
 \omega_t=w\alpha w^{-1},
\]
where $\alpha\in\{x,y,z\}$ is the letter attached to the central crossing of
$L_t$.  The crossings on the two sides of the central crossing occur in the
same symmetric pairs for $L_t$ and for $\overline L_t$; hence they still
contribute the outer factors $w$ and $w^{-1}$.  The effect of the endpoint
completion is on the central crossing: the central letter $\alpha$ is replaced
by $\alpha^{-1}$.  Therefore the non-endpoint part of the reading of
$\overline L_t$ is $
 w\alpha^{-1}w^{-1}.$
 Since $
 (w\alpha w^{-1})^{-1}=w\alpha^{-1}w^{-1},$
this non-endpoint part is precisely $\omega_t^{-1}$.  Combining this with the
initial endpoint factor $xyz$ gives
\[
 \overline{\omega}_t=xyz\omega_t^{-1}.
\]
\end{proof}

\begin{example}\label{ex:endpoint-completed-word-two-fifths}
For $t=\frac{2}{5}$, Lemma~\ref{lem:endpoint-completed-word} gives
\[
\overline{\omega}_{\frac{2}{5}}=xyzyzy^{-1}z^{-1}y^{-1}x^{-1}yzyz^{-1}y^{-1}.
\]
This edge-reading construction is illustrated in Figure~\ref{fig:extended-word}.
\end{example}

\begin{figure}[ht]
    \centering
    \includegraphics[scale=0.11]{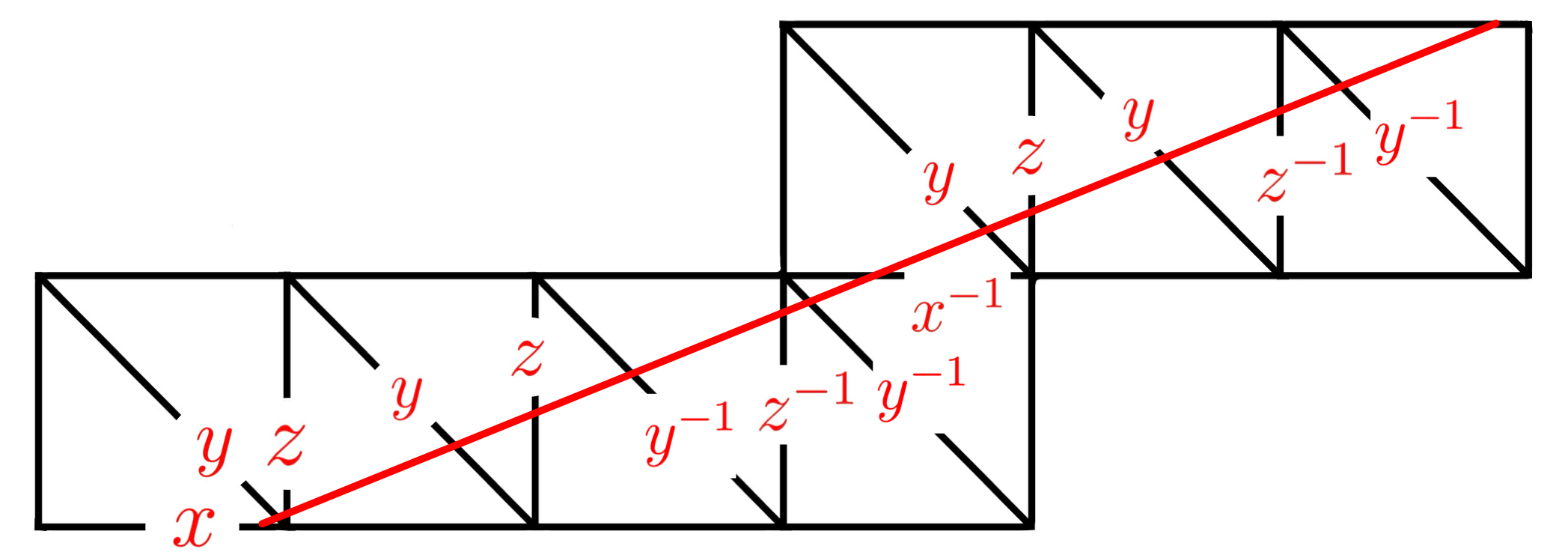}
    \caption{Letters assigned to edges in $\overline L_{\frac{2}{5}}$}
    \label{fig:extended-word}
\end{figure}

Substituting the matrices $X,Y,Z$ of Section~\ref{sec:mm-matrices} into this word gives the \emph{endpoint-completed matrix}
\begin{equation}\label{eq:extended-mm-matrix-definition}
\overline M_t
:=\overline{\omega}_t\big|_{x\mapsto X,\;y\mapsto Y,\;z\mapsto Z}.
\end{equation}
The following theorem gives the direct passage from $\overline M_t$ to the generalized Cohn matrix.  The important point is that one does not need to first form the strongly admissible sequence.  The entries of the generalized Cohn matrix are already present in $\overline M_t$; only a fixed sign correction is required.

\begin{theorem}[Obtaining the generalized Cohn matrix]\label{thm:gc-from-extended-word}
For every positive reduced fraction $t\in(0,\infty)\cap\mathbb Q$, one has
\begin{equation}\label{eq:gc-from-extended-by-switch}
C_t=
\begin{bmatrix}-1&0\\0&1\end{bmatrix}
\overline M_t
\begin{bmatrix}1&0\\0&-1\end{bmatrix}.
\end{equation}
\end{theorem}

\begin{proof}
By Lemma~\ref{lem:endpoint-completed-word}, the endpoint-completed word satisfies $\overline{\omega}_t=xyz\omega_t^{-1}$.  Hence
\[
\overline M_t=XYZM_t^{-1}.
\]
A direct multiplication of the matrices in Section~\ref{sec:mm-matrices} gives
\begin{equation}\label{eq:explicit-xyz-product}
XYZ=
\begin{bmatrix}
-1&K\\
0&-1
\end{bmatrix}.
\end{equation}
On the other hand, Theorem~\ref{thm:mm-matrix-entries} gives
\[
M_t^{-1}=
\begin{bmatrix}
-(u_t+k_t)&\displaystyle\frac{u_t^2+k_tu_t+1}{m_t}\\[6pt]
-m_t&u_t
\end{bmatrix}.
\]
Therefore
\[
\begin{aligned}
\overline M_t
&=
\begin{bmatrix}
-1&K\\
0&-1
\end{bmatrix}
\begin{bmatrix}
-(u_t+k_t)&\displaystyle\frac{u_t^2+k_tu_t+1}{m_t}\\[6pt]
-m_t&u_t
\end{bmatrix} \\[4pt]
&=
\begin{bmatrix}
u_t+k_t-Km_t&\displaystyle Ku_t-\frac{u_t^2+k_tu_t+1}{m_t}\\[6pt]
m_t&-u_t
\end{bmatrix}.
\end{aligned}
\]
Set
\[
A_t:=Km_t-k_t-u_t,
\qquad
B_t:=\frac{Km_tu_t-k_tu_t-u_t^2-1}{m_t}.
\]
Then the last display becomes
\[
\overline M_t=
\begin{bmatrix}
-A_t&B_t\\
m_t&-u_t
\end{bmatrix}.
\]
Since
\[
C_t=
\begin{bmatrix}
A_t&B_t\\
m_t&u_t
\end{bmatrix},
\]
the fixed sign correction in \eqref{eq:gc-from-extended-by-switch} gives $C_t$.
This proves the theorem.
\end{proof}

\begin{example}\label{ex:gc-from-extended-two-fifths}
Continue with $(k_1,k_2,k_3)=(1,2,0)$, $\sigma=\mathrm{id}$, and
$t=\frac{2}{5}$.  From the endpoint-completed word computed above,
\[
\overline{\omega}_{\frac{2}{5}}
=xyzyzy^{-1}z^{-1}y^{-1}x^{-1}yzyz^{-1}y^{-1}.
\]
Substituting the same matrices $X,Y,Z$ as in
Example~\ref{ex:mm-matrix-entries-two-fifths}, we obtain
\[
\begin{aligned}
\overline M_{\frac{2}{5}}
&=XYZYZY^{-1}Z^{-1}Y^{-1}X^{-1}YZYZ^{-1}Y^{-1}\\
&=\begin{bmatrix}-47431&11127\\8227&-1930\end{bmatrix}.
\end{aligned}
\]
The sign correction in Theorem~\ref{thm:gc-from-extended-word} gives
\[
\begin{bmatrix}-1&0\\0&1\end{bmatrix}
\overline M_{\frac{2}{5}}
\begin{bmatrix}1&0\\0&-1\end{bmatrix}
=\begin{bmatrix}47431&11127\\8227&1930\end{bmatrix}.
\]
This is the same matrix $C_{\frac{2}{5}}$ as in
Example~\ref{ex:gc-cf-two-fifths}.  In particular,
the $(2,1)$-entry is the GM number $8227$, and the $(2,2)$-entry is the
characteristic number $1930$.
\end{example}

\begin{remark}
The two diagonal sign matrices in \eqref{eq:gc-from-extended-by-switch} are independent of $t$, of $(k_1,k_2,k_3)$, and of $\sigma$.  Hence the passage from $\overline M_t$ to the generalized Cohn matrix is uniform over the whole Farey tree.  For $t\in(0,\infty)\cap\mathbb Q$, combining Theorem~\ref{thm:gc-from-extended-word} with Theorem~\ref{thm:gc-from-strongly-admissible} also gives
\[
F_{s_{k_1,k_2,k_3,\sigma}(t)}=
\begin{bmatrix}-1&0\\0&1\end{bmatrix}
\overline M_t
\begin{bmatrix}1&0\\0&-1\end{bmatrix}.
\]
\end{remark}

\bibliographystyle{amsplain}
\bibliography{myrefs}

\end{document}